\documentclass[11pt]{amsart}
\usepackage{graphicx} 
\usepackage[utf8]{inputenc}
\usepackage{amsmath}
\usepackage{xcolor}
\usepackage{amssymb}
\usepackage{amsmath,stackengine}
\usepackage{amsfonts}
\usepackage{amstext}
\usepackage{amsthm}
\usepackage[english]{babel}
\usepackage[autostyle]{csquotes}
\usepackage[super]{nth}
\usepackage{hyperref}
\usepackage[margin=1.4in]{geometry}
\usepackage{txfonts}
\usepackage[normalem]{ulem}
\usepackage{parskip}

\newcommand{\mycomment}[1]{}

\newtheorem{theorem}{Theorem}[section]
\newtheorem{proposition}[theorem]{Proposition}
\newtheorem{lemma}[theorem]{Lemma}
\newtheorem{corollary}[theorem]{Corollary}
\newtheorem{definition}[theorem]{Definition}
\newtheorem*{remark}{Remark}

\newtheorem*{example}{Example}

\newcommand{\rn}{\mathbb{R}^{n}}
\newcommand{\re}{\mathbb{R}}

 \numberwithin{equation}{section}

\hypersetup{
    colorlinks=true,
    linkcolor= blue,
    citecolor=blue,
    pdfpagemode=FullScreen,
    }
{%
\setcounter{enumi}{0}

\begin{enumerate}}%
{\end{enumerate} }

\title[LARGE SOLUTIONS OF QUASI-LINEAR EQUATION ON INFINITE CYLINDERS]{BOUNDARY BLOW-UP SOLUTIONS OF SECOND ORDER QUASILINEAR EQUATION ON INFINITE CYLINDERS}
\author[I. Chowdhury]{Indranil Chowdhury}
\address{\parbox{.8\linewidth}
{{\textbf{I. Chowdhury}}\medskip \\
Indian Institute of Technology - Kanpur, India \medskip}}
\curraddr{}
\email{indranil@iitk.ac.in}

\author[N.  Dattatreya]{N N Dattatreya}
\address{\parbox{.8\linewidth}
{{\textbf{N.  Dattatreya}}\medskip \\
Indian Institute of Technology - Kanpur, India \medskip}}
\curraddr{}
\email{dattatreya21@iitk.ac.in}
\date{}

\smallskip


\begin{document}

\begin{abstract}

This article studies large solutions for a class of Quasi-linear equations,  including the p-Laplace operator,  on the \textit{infinite cylindrical domain}.
We study the existence of `weak' large solutions on such a cylinder by the convergence of large solutions on finite cylinders and observe that any such solution coincides with the large solution on its cross-section. We also address the uniqueness results for equations involving the p-Laplacian. 

\end{abstract}
\maketitle
\section{Introduction}

The focus of this paper is the existence and uniqueness of solutions for the following problem
\begin{align}\label{eqn_mainQ}
    \begin{cases}
    \text{div}(Q(|\nabla u|)\nabla u)= f(u)\quad &\text{in}~ S\\
        u(x)\to \infty &\text{as } \text{dist}(x,\partial S)\to 0,
    \end{cases}
\end{align}
on the infinite cylinder  $S=B^{n-1}_{1}(0)\times \mathbb{R}$, where $B^{n-1}_{1}(0)$ is unit ball in $\mathbb{R}^{n-1}$, $Q:(0,\infty)\to (0,\infty)$ is a $C^1$ function such that $Q(0^{+})=0$ and $(Q(r)r)'>0$ for $r>0$, and $f:[0,\infty)\to [0,\infty)$ increasing, continuous function such that $f(r)\to \infty$ as $r\to\infty$.

 In general, the existence of a solution for \eqref{eqn_mainQ} is connected with the `Keller-Osserman' condition, when $Q(r)=1$ it reads as
\begin{align*}
    \int^{\infty}\frac{ds}{\sqrt{F(s)}}<\infty,
\end{align*}
where $F$ is the primitive of $f$.
 Such solutions, whenever they exist, are called \textit{Large Solutions}. 
The literature also refers to such solutions as \textit{Explosive solutions} or \textit{ Boundary blow-up solutions}. The study of these problems traces back to L. Biberbach’s 1916 work \cite{biberbach}, which addressed well-posedness for the case $Q(r)=1$ and $f(t)=e^{t}$ on a planar domain, along with the conformal mapping representation of solutions on simply connected domains. A significant milestone emerged in 1957 with the independent contributions of Keller and Osserman. Keller’s work \cite{osti_4157678, Keller1957} originated in electrodynamics, examining the equilibrium of charged gas particles in a container, while Osserman’s \cite{Osserman1957} approached the problem from a geometric perspective, focusing on ball geometry and Gauss curvature. Later, J. M. Lasry and P. L. Lions \cite{lions} explored boundary blow-up solutions involving the Laplace operator in the context of stochastic control.

\medskip
The well-posedness of large solutions for semi-linear elliptic equations on bounded domains is now well established, the related developments are in e.g. \cite{bandlemarcus4, MR1240799, essen, bandlegiarrusso, Amandine1997, veronmarcus} and the references therein. 
 For example, \cite{Amandine1997} establishes the existence of blow-up solutions for two classes of locally Lipschitz non-decreasing nonlinearities: one vanishing at zero and the other strictly positive. Further,  \cite{essen,bandlegiarrusso} deal with $-\Delta u=au-b(x)f(u)$ and  $\Delta u\pm |\nabla u|^{q}=p(x)u^{\gamma}$, respectively. Whereas \cite{Costin, Lieberman} analyze symmetric solutions on a ball for nonlinearities of the form $f(t)=\Tilde{f}(t)-\lambda t$. 

Another interesting aspect of large solutions is their asymptotic behavior near the boundary. For results in specific domains—such as the unit ball and the half-cylinder—we refer to \cite{Costin, Payne}, respectively. The study of such asymptotics is extensive, as it depends not only on the differential operator and the forcing term but also on the geometry of the domain's boundary. For related works on bounded domains, see \cite{bandlemarcus4, Payne, McKenna, Bandle1998OnSE, cbandle, bandlesurvay, Shuibo2010, Lieberman}.  
The case of forcing terms of the form $u^{q}$ and $e^{u}$ is treated in \cite{cbandle, bandlesurvay}, while \cite{Shuibo2010} establishes the precise blow-up rate for solutions to $\Delta u=b(x)f(u)$ on smooth domains. Second-order effects are examined in \cite{Bandle1998OnSE}, and gradient asymptotics are derived in \cite{bandlemarcus3},  and in \cite{SBuccheri2019} for the problem
$-\Delta u+|u|^{p-2}u=f$. The influence of boundary curvature on large solutions is explored in \cite{Bandlemarcus1, Pino}. For non-smooth domains, we refer to \cite{veronmarcus}, and for fractal domains, to \cite{matero}. Works on logistic equations include \cite{Letelier2001, FLORICA2002}, while the existence of large solutions to the Monge-Ampère equation is addressed in \cite{Materononlinear, Mohammed2007, Zhijun2018, Zhang2020}. Finally, an interesting application of large solutions is discussed in \cite{kichenassamy}.


The well-posedness and asymptotic behaviour of large solutions for the 
p-Laplacian on bounded domains has been  studied in recent decades; see \cite{materopaper, materothisis, Zongming2007,FLeoni2000} and the references therein.
For general quasilinear equations of the form $-div(Q(|\nabla u|\nabla u))+\lambda f(u)=g$ on bounded domains, analogous results are established in \cite{diaz1993, Du2003}. Similarly, fully nonlinear problems such as $u-\Delta_{p}u+\beta |\nabla u|^{q}=f$ are investigated in \cite{Buccheri}. 


The study of asymptotics of sequences of solutions defined on finite cylinders becoming unbounded, including asymptotics of energies of solutions as well as eigenvalues and eigenvectors, has been studied extensively for various operators.  A Detailed background for the same can be found in \cite{Chipotbook,ChipotRougirel2008,ChipotYeressian2008} and references therein. We refer to \cite{ChipotXie2004, ChipoXie2008} for such studies in case of the p-Laplacian on cylinders becoming unbounded, \cite{Chipot2024} for $(p,q)$-Laplacian and \cite{WangZhou2021} for a nonlinear problem.
An eigenvalue problem for p-Laplacian can be found in \cite{RoyLucaFiroj2021}, nonlinear eigenvalue problem in \cite{RamaHarriPro2014}. For results on anisotropic p-Laplacian, we refer to \cite{Jana2024}. Further, parabolic quasilinear and nonlinear problems are dealt in \cite{ChipotXie2005, Senoussi2009}, respectively.


\mycomment{

\textcolor{blue}{ If $f=0$, the Keller Osserman condition fails, rendering the non-existence of large solutions. Blowing up at the boundary isn't an inherent property of this local operator but is imposed by the boundary condition. Nonlinearity tries to avert the solution from blowing up in the form of a local bound on the solution, boundary condition encourages the solution to blow up. The border at which one takes over the other is the Keller-Osserman condition.}  this goes!}


The non-existence of boundary blow-up solutions on the whole space $\mathbb{R}^N$, given in \cite{Keller1957}, raises a natural question: Can similar results hold when some-but not all-directions are unbounded? For semi-linear equations ($Q(r) = 1$), an answer is provided in \cite{bandle2013large}. The authors establish the existence of such solutions on infinite cylindrical domains as a limit of large solutions on a sequence of finite cylinders converging to the infinite cylinder. They further show that these solutions coincide with the large solution on the corresponding slice of the cylinder. Intuition for the latter emerges from the one-dimensional case in $\mathbb{R}$, where solutions $v_\ell$ on $(-\ell, \ell)$ converge locally uniformly to zero as $\ell \to \infty$.
\medskip
 
 We aim to achieve similar existence results of \eqref{eqn_mainQ} by the approximation of the large solutions on finite cylinders becoming unbounded in particular direction(s). With a general assumption on $Q$, that is satisfied by the p-Laplacian, we prove the following statements :
 \begin{enumerate}
     \item [{\small{(i)}}]  There exists $u\in W^{1,\infty}_{loc}(S)$ that solves \eqref{eqn_mainQ}, when $f$ satisfy \textit{Keller Osserman} type condition given in \ref{A1} bellow.   
     \item [{\small{(ii)}}] $u$ is independent of the last variable and solves a corresponding problem on $B^{n-1}_{1}(0)$ for a fixed last variable. 
     \item [{\small{(iii)}}] $u$ is unique in solving \eqref{eqn_mainQ} for $Q(r)=r^{p-2}, p\geq 1$, with an additional compatible condition on $f$.
 \end{enumerate}
  The results are then generalised for the infinite cylinder of the form $S(\omega)=\omega\times \mathbb{R}^{n-m}$, where $\omega\subset\mathbb{R}^{m}$ for $1\leq m<n$, in section \ref{subsection-ext}.

 In the proof, we observe that ${u_{\ell}}$, the sequence of large solutions on the finite cylinders $B^{N-1}_{1}(0) \times (-\ell, \ell)$, forms a decreasing sequence of non-negative functions. We thus define $u$ as its limit. In the case of the Laplace operator, the local uniform convergence of $u_{\ell}$ to $u$ ensures that the limit commutes with the operator, as demonstrated in \cite{bandle2013large}, where the authors work with pointwise solutions.

When $p \neq 2$, the solution is weak, requiring careful treatment of the limit and integral interchange. Two primary challenges arise: Nonlinearity in the gradient term, which obstructs direct passage of the limit. Boundary conditions restrict the solution from serving as a test function, even after scaling or translation.
Thus, our approach relies on monotonicity or Browder-Minty-type arguments to establish the result for general $Q$. However, for the special case $Q(r) = r^{p-2}$, we provide an alternative proof using a vector inequality and a meticulously chosen test function.

In Section \ref{section 3}, we adapt techniques from \cite{materopaper} to regulate blow-up behaviour. By constructing a radial, blow-up supersolution on balls (utilising results from Section \ref{Section 2}), we derive a local uniform bound for solutions of \eqref{eqn_mainQ}. Additionally, for $Q(r) = r^{p-2}$, certain interior integral estimates in this section follow directly from the results in \cite{FLeoni2000}.

The remainder of this paper is organised as follows:

Section \ref{Section 2} analyses the one-dimensional case of the problem, providing crucial intuition about the structure of large solutions on infinite cylinders. Several arguments developed here are subsequently employed in Section \ref{section 3}.
Section \ref{section 3} presents fundamental definitions and establishes key properties. We derive a local uniform bound for weak solutions, and we prove a local $L^p$ bound for their gradients for the specific case where $Q(r) = r^{p-2}$, both in bounded domains.
 The main results are presented in Section \ref{section 4} and \ref{section 5}. In Section \ref{section 5}, we prove (i) and (ii) above, using monotonicity methods. Whereas in Section \ref{section 4} we give a different proof for the special case $Q(r)=r^{p-2}$. Also, at the end of Section \ref{section 5}, we discuss results in more general cylindrical domains.

\section{Analysis In One-Dimension}\label{Section 2}
This section presents a one-dimensional analysis of a problem more general than \eqref{eqn_mainQ}. The results parallel those obtained in \cite{bandle2013large} for a linear operator and align with the conjectures made for equations of the form \eqref{eqn_1d}.  This analysis offers insight into the asymptotic behaviour in higher dimensions.
 Furthermore, the solution constructed here serves as a tool to establish local bounds for solutions in higher dimensions. 

For $\ell>0$, we consider
\begin{align}
 (A(v'))'= f(v) \quad &\text{in} \quad (-\ell,\ell) \label{eqn_1d}\\
   v(x)\rightarrow \infty \quad\quad &\text{as} \quad x\rightarrow \pm \ell. \label{eqn_1d_bdry}
   \end{align}
Where $A:\mathbb{R}\rightarrow \mathbb{R}$ is a continuously differentiable, increasing function with $A(0)=0$ and $A'(r)>0~\forall ~r>0$; further assumptions on $A$, required for the analysis, will be discussed in due course. In addition, $f:[0,\infty)\rightarrow [0,\infty)$ is a continuous function which is increasing and $f(r)>0$ for $r>0$ and $f(t)\to \infty$ as $t\to \infty$.
\begin{example} One may consider,
    \begin{enumerate}
        \item [\textit{i)}] $A(r)=|r|^{p-2}r$, which is the p-Laplacian.
        \item [\textit{ii)}] $A(r)=\frac{r}{\sqrt{1+r^{2}}}$.
    \end{enumerate}
\end{example}

\makeatletter
\newcommand{\myitem}[1]{%
\medskip \item[#1]\protected@edef\@currentlabel{#1}%
} 

\begin{enumerate}
    \myitem{$\textbf{(A1)}$} \label{A1} \textit{We assume the following Keller-Osserman condition: For any $r>0$ }
    \begin{align*}
        \Psi_{A}(r):=\int_{r}^{\infty}\frac{ds}{B^{-1}\{F(s)\}}<\infty,
    \end{align*}
     where \begin{align}\label{dfn F and B}
    \displaystyle F(x)=\int_{0}^{x}f(s)\ ds,\quad B(x)=\int_{0}^{x}A'(s)s\ ds.
\end{align} 
\end{enumerate}
The condition is also denoted as $\int^{\infty}\frac{ds}{B^{-1}\{F(s)\}}<\infty$. 
\begin{example}
     When $A(r)=|r|^{p-2}r$  we get $ \displaystyle \Psi_{p}(r)=\biggl(1-\frac{1}{p}\biggl)^{\frac{1}{p}}\int_{r}^{\infty}\frac{dx}{F(s)^{\frac{1}{p}}}<\infty$
and for  $f(t)=t^{q}$ \ref{A1} becomes $\displaystyle \int^{\infty}\frac{dx}{x^{\frac{q+1}{p}}}<\infty.
$
Therefore, $q>p-1$. 
 \end{example}
For a $v_{0}\in (0,+\infty)$, we consider the following IVP:
\begin{align}\label{eqn_1d_ivp}
    \begin{cases}
        (A(w'))'= f(w) \quad &\forall~ x>x_{m}\\
        w'(x_{m})=0 \quad \text{and} &w(x_{m})=v_{0}. 
    \end{cases}
\end{align}
Multiplying both sides by $w'$, integrating between $x_{m}<x$,
and by a change of variable 
\begin{align*}
    \int_{0}^{w'(x)}A'(s)s\ ds=\int_{v_{0}}^{w(x)}f(s)\ ds.
\end{align*}
Then we have by the above notation, $B(w'(x))=F(w(x))-F(v_{0}).$
 Which implies 
 \begin{align}\label{eqn_one dimensional analysis}
     \frac{w'(x)}{B^{-1}\{F(w(x))-F(v_{0})\}}=1.
 \end{align}
 Integrating from $x_{m}$ to $x$ with the change of variable applied to the left-hand side, we get,
 \begin{equation}\label{eqn_for_KO}
     \int_{v_{0}}^{w(x)}\frac{ ds}{B^{-1}\{F(s)-F(v_{0})\}}=x-x_{m}.
 \end{equation}
Note that  the right side of \eqref{eqn_for_KO} tends to $\infty$ as  $x\to \infty$,  whereas by assuming \ref{A1} the left side remains finite  as long as $v_{0}=v(x_{m})>0$. Therefore $w$ should blow up at a finite value of $x$, say $x=\ell(v_{0})$, where $(x_{m}, \ell(v_{0}))$ is the maximal interval of existence. We have 
\begin{equation}\label{eqn_KO_ell}
 \ell(v_{0})=x_{m}+\int_{v_{0}}^{\infty}\frac{ds}{B^{-1}\{F(s)-F(v_{0})\}} .  
\end{equation}
Tracing the steps back, any $w$ defined implicitly as in \eqref{eqn_for_KO} solves \eqref{eqn_1d_ivp}.
To see the uniqueness of solution to \eqref{eqn_1d_ivp}, assume $w_{1}$ and $w_{2}$ are two solutions, then by \eqref{eqn_for_KO}
\begin{align*}
    \int_{w_{1}(x)}^{w_{2}(x)}\frac{ ds}{B^{-1}\{F(s)-F(v_{0})\}}=0,
\end{align*}
implies $w_{1}=w_{2}$ for all $x>x_{m}$, as the integrand is positive.
To this end, we have that $w$ is the only solution to \eqref{eqn_1d_ivp} and is implicitly given by \eqref{eqn_for_KO}.

 Similarly, let $\Tilde{w}$ solve the counterpart
\begin{align*}
    \begin{cases}
        (A(w'))'=f(w) \quad &\forall~ x<x_{m}\\
        w'(x_{m})=0 \quad &w(x_{m})=v_{0},
    \end{cases}
\end{align*}
we can construct a function 
\begin{equation*}
    v(x)=\begin{cases}
        w(x) \quad \text{if } x\geq x_{m}\\
        \Tilde{w}(x) \quad \text{if } x\leq x_{m}
    \end{cases}
\end{equation*}that solves \eqref{eqn_1d} and \eqref{eqn_1d_bdry} uniquely on $(-\ell(v_{0}),\ell(v_{0}))$ and is given implicitly as in \eqref{eqn_for_KO}.

 Further whenever $x<x_{m}$ by reflection, $2x_{m}-x>x_{m}$ and so $v(2x_{m}-x)$ solve \eqref{eqn_1d_ivp}, and vice versa. We deduce that $v(2x_{m}-x)=v(x)$ for all $x$ for which $v$ is defined.
 Since the domain we are looking for is symmetrical about $0$, we must consider $x_{m}=0$. Thus, $v$ blows up at $\pm \ell(v_{0})$, $v(-x)=v(x)$ for all $x\in (-\ell(v_{0}),\ell(v_{0}))$ and, is implicitly given by
\begin{equation}\label{eqn_1d_sln_implicit}
     \int_{v_{0}}^{v(x)}\frac{ ds}{B^{-1}\{F(s)-F(v_{0})\}}=x \quad \forall~ x\in (-\ell(v_{0}),\ell(v_{0})).
 \end{equation}
It is easy to see that $v$ is decreasing in $(-\ell(v_{0}),0)$ and increasing in $(0,\ell({v_{0}}))$ as $v''>0$ by chain rule. Thus we have 
\begin{align}\label{eqn:v_0}
\displaystyle v_{0}=\min_{x\in (-\ell(v_{0}),\ell(v_{0}))} v(x)=v(x_{m})
 \end{align}

 Before going further, we consider some necessary assumptions regarding the integrability of  $\int_{0}^{r}\frac{ds}{B^{-1}\{F(s)\}}$ for any $r\in (0,\infty)$, which is, denoted by $\int_{0^{+}}\frac{ds}{B^{-1}\{F(s)\}}$.

\begin{enumerate}
        \myitem{$\textbf{(A2)}$}\label{A2} $\int_{0^{+}}\frac{ds}{B^{-1}\{F(s)\}}=\infty$, \textit{which is called the $Osgood~ Condition$.}
    \myitem{$\textbf{(A3)}$}\label{A3} $\int_{0^{+}}\frac{ds}{B^{-1}\{F(s)\}}<\infty$.
\end{enumerate}
First, we prove a lemma that gives the non-existence of a large solution on $\re$.

\begin{lemma}\label{lemma phi_0 neq 0}
Assume \ref{A2} and $v_0$ as in \eqref{eqn:v_0}, then $v_{0}\ne 0$.
\end{lemma}
\begin{proof}
Assume the contrary, then $F(v_{0})=0$. The equation \eqref{eqn_1d_ivp} in this case has a unique solution $w=0$, because, if $w(x)>0$ for some $x$, take $x_{0}=\min\{x~|~w(x)>0\}$, then $x_{0}\geq0$, $w(x_{0})=0$ and $(x_{0},\infty)\subset \{x~|~w(x)>0\}$.
\\
By \eqref{eqn_1d_sln_implicit}, with the same argument as between \eqref{eqn_1d_ivp} and \eqref{eqn_for_KO}, and integrating between $x<y$ we get,
\begin{align*}
   \displaystyle  \int_{w(x)}^{w(y)}\frac{ds}{B^{-1}\{F(s)\}}=y-x.
\end{align*}
Letting $x \to x_{0}$, we get a contradiction to \ref{A2} as the left-hand side would then become infinite, whereas the right-hand side remains finite. Hence $v_{0}\ne 0$.
\end{proof}

\begin{theorem}
With \ref{A1} and \ref{A2}, for any $\ell>0$ there exists a unique solution $v_{\ell}$ to \eqref{eqn_1d},\eqref{eqn_1d_bdry}
    with the following properties:
    \begin{itemize}
        \item[\textit{(i)}] $v_{\ell}$ is convex.
        \item [\textit{(ii)}] $v_{\ell}\to 0$ uniformly on any compact subset of $\mathbb{R}$, as $\ell\to \infty$.
    \end{itemize}
\end{theorem}
\begin{proof}
We prove the existence of $v_{\ell}$ for all $\ell>0$ by showing that $\ell(v_{0}):(0,+\infty)\to(0,+\infty)$, as a function of $v_{0}$, is onto, and uniqueness by showing that it is one-one.  To show that it is onto, we show that it is continuous, $\displaystyle \lim_{v_{0}\to\infty}\ell(v_{0})=0 $ and $\displaystyle \lim_{v_{0}\to 0}\ell(v_{0})=\infty $.\\
We have defined $v_{\ell}$ above. As we have observed by chain rule $v''>0$, hence $(i)$ holds. \\[0.6em]
\noindent \textit{\underline{Step 1:} }$\ell(v_{0})$ is onto.\\ 
 By \eqref{eqn_KO_ell} and change of variable, we get,
\begin{align*}
    \ell(v_{0})=\int_{0}^{\infty}\frac{ds}{B^{-1}\{F(s+v_{0})-F(v_{0})\}},
\end{align*}  
hence $\ell$ is a continuous function of $v_{0}$. Differentiating the denominator of the integrand with respect to $v_{0}$
\begin{equation}\label{eqn_diff_wrt_phi0}
   \frac{d}{dv_{0}}\{F(s+v_{0})-F(v_{0})\}=f(s+v_{0})-f(v_{0})\geq 0, 
\end{equation}
implies that $F(s+v_{0})-F(v_{0})$ is increasing with respect to $v_{0}$ for a fixed $s$. Since $B$ is increasing, $B^{-1}\{F(s+v_{0})-F(v_{0})\}$ is increasing in $v_{0}$ and hence $\ell(v_{0})$ is decreasing in $v_{0}$.
 For $s>0$
 \begin{align*}
     F(s+v_{0})-F(v_{0})=\int_{v_{0}}^{s+v_{0}}f(t)\ dt\geq \int_{v_{0}}^{s+v_{0}}f(v_{0})\ dt =sf(v_{0}),
 \end{align*}
thus $F(s+v_{0})-F(v_{0})\to +\infty$ as $v_{0}\to \infty$ because $f(r)\to\infty$ when $r\to\infty$.  We claim that $B(x)\to +\infty$ as $x\to+\infty$. With this we have,  $B^{-1}\{F(s+v_{0})-F(v_{0})\} \to \infty$ as $v_{0}\to \infty$ and so $\displaystyle \lim_{v_{0}\to\infty}\ell(v_{0})=0 $.\\
To prove the claim, we notice that $A'>\alpha$ for some $\alpha>0$ as it is continuous. Hence,
\begin{equation*}
\displaystyle\lim_{x\to+\infty}B(x)\geq \lim_{x\to+\infty}\int_{0}^{x}\alpha s\ ds= +\infty.
\end{equation*}
 On the other hand, as $v_{0}\to 0$, $F(s+v_{0})-F(v_{0})\to F(s)$, thus $B^{-1}\{F(s+v_{0})-F(v_{0})\}$ decreases to $ B^{-1}\{F(s)\}$ by \eqref{eqn_diff_wrt_phi0}. By Monotone convergence theorem and \ref{A2} we then conclude that $\displaystyle \lim_{v_{0}\to 0}\ell(v_{0})=\infty $. \\
 
\noindent  \textit{\underline{Step 2:}} To show that $\ell$ is strictly decreasing with respect to $v_{0}$.\\
 Since $f$ is monotonically increasing, $f(s+v_{0})-f(v_{0})>0$ for large $s$, by \eqref{eqn_diff_wrt_phi0} we get that $F(s+v_{0})-F(v_{0})$ is strictly increasing for large $s$ and so
 \begin{align*}
     \frac{1}{B^{-1}\{F(s+v_{0})-F(v_{0})\}}
 \end{align*} is strictly decreasing, as a result  $\ell$ is strictly decreasing.  That proves the claim.\\
 
 \noindent  \textit{\underline{Step 3:}} Uniqueness for solution of \eqref{eqn_1d},\eqref{eqn_1d_bdry}.\\
  Say $v_{1}$ and $v_{2}$ are two solutions. If $v^{1}_{0}=v^{2}_{0}$, uniqueness follows from the uniqueness of \eqref{eqn_1d_ivp} and its counterpart. If $v^{1}_{0}<v^{2}_{0}$, then, $\ell(v^{1}_{0})<\ell(v^{2}_{0})$. They can not simultaneously solve \eqref{eqn_1d_bdry}.\\
  
  \noindent  \textit{\underline{Step 4:}} Proof of \textit{(ii)}.\\
Integrating \eqref{eqn_one dimensional analysis} from $x>x_{m}$ to $\ell(v_{0})$ for $w=v_{\ell}$, we get,
\begin{equation*}
  \int_{v_{\ell}(x)}^{\infty}\frac{ ds}{B^{-1}\{F(s)-F(v_{0})\}}=\ell(v_{0})-x . 
\end{equation*}
 When $\ell=\ell(v_{0})\to \infty$, due to Keller-Osserman and Osgood conditions $v_{\ell}(x)\to 0$. For any compact set $K$ in $\re$ we can find $a>0$ in $\re$ such that $K\subset [-a,a]$ in $\mathbb{R}$, since $v_{\ell}(a)\to 0$, $v_{\ell}\to 0$ uniformly on $K$ as $\ell\to\infty$.
\end{proof}
 
 \begin{corollary}{(Non-existence)}
 There is no large solution solving $(A(v'))'= f(v)$ on $\re$.
 \end{corollary}
 \begin{proof}
 From \eqref{eqn_KO_ell}, there is a unique $\ell(v_{0})<+\infty$ whenever $v_{0}>0$ and from the last proof $\displaystyle\lim_{v_{0}\to 0}\ell(v_{0})=+\infty$. Given that $f\geq 0$, if $v$ is such a large solution on $\re$, then $v_{0}=\displaystyle\min_{x\in \re}v(x)=0$. A contradiction to Lemma \ref{lemma phi_0 neq 0}.
 \end{proof}
 
\begin{theorem}
Assume \ref{A1} and \ref{A3}. Fix
\begin{align*}
    L=\int_{0}^{\infty}\frac{ds}{B^{-1}\{F(s)\}}
\end{align*} then there is a unique solution $\Tilde{v}_{\ell}$ to \eqref{eqn_1d},\eqref{eqn_1d_bdry}. Also $\Tilde{v}_{\ell}\to 0$ uniformly on any compact set, moreover for $\ell>L$, $\Tilde{v}_{\ell}$ develops dead core.
\end{theorem}
\begin{proof}
    Arguments are similar for $\Tilde{v}_{0}>0$, except when $\Tilde{v}_{0}\to 0$, $F(\Tilde{v}_{0}+s)-F(\Tilde{v}_{0})\to F(s)$ and so by Monotone convergence theorem
    \begin{align*}
        \lim_{\Tilde{v}_{0}\to 0}\ell(\Tilde{v}_{0})=\int_{0}^{\infty}\frac{ds}{B^{-1}\{F(s)\}}=L<\infty.
    \end{align*}
  For the case when $\Tilde{v}_{0}=0$ and $\ell\geq L$, we construct the solution using the solution $v_{L}$ from the previous theorem, and before doing so, we see how it should fit. \\
  If an increasing $v$ should solve \eqref{eqn_1d_ivp}, let $x_{0}=min\{x>x_{m}~|~v(x)>0\}$, define $v$ for $x>x_{m}$ implicitly by
 \begin{align*}
     \int_{0}^{v(x)}\frac{ds}{B^{-1}\{F(s)\}}=x-x_{0} \quad \quad \forall ~ x>x_{0},
 \end{align*}
  when $x\to L+x_{0}$ we have $v(x)\to \infty$. By taking $x_{0}=\ell-L$ we have a solution for \eqref{eqn_1d_ivp} that blows up at $\ell$. i,e,. there is a unique solution $\Tilde{v}_{\ell}$ solving \eqref{eqn_1d_ivp} such that $\Tilde{v}_{\ell}(x)=0 ~\forall~ x\in (x_{m},\ell-L]$ and is given implicitly by 
 \begin{align*}
     \int_{0}^{\Tilde{v}_{\ell}(x)}\frac{ds}{B^{-1}\{F(s)\}}=x-\ell+L \quad \forall x\in [\ell-L,\ell).
 \end{align*}
  By the symmetric construction, we'll have a unique solution to \eqref{eqn_1d}, \eqref{eqn_1d_bdry} of the form
  \begin{align*}
      \Tilde{v}_{\ell}(x)=0 \quad \forall~x\in [L-\ell,\ell+L],
  \end{align*}
 And is implicitly defined by
 \begin{align*}
     \int_{0}^{\Tilde{v}_{\ell}(x)}\frac{ds}{B^{-1}\{F(s)\}}=x-\ell+L \quad \forall x\in (-\ell, L-\ell)\cup (\ell-L, \ell).
 \end{align*}
 Thus, $\Tilde{v}_{\ell}$ develops dead core in $ [L-\ell,\ell+L]$. Uniform convergence of $\Tilde{v}_{\ell}$ follows from the uniform convergence of $v_{L}$.
\end{proof}

\section{Preliminary Results for higher dimension }\label{section 3}

In this section, we present results concerning large solutions on bounded domains in general dimension $n$. Similar results have been previously established in \cite{diaz1993,materopaper}.

 Let $\Omega \in \rn$, $n\geq 2$ be bounded, $Q:(0,\infty)\to (0,\infty)$ is a $C^1$ function such that $Q(0^{+})=0$ and $(Q(r)r)'>0$ for $r>0$ we consider the problem
\begin{equation}\label{eqn_noboundary condition} 
    \text{div}(Q(|\nabla u|)\nabla u)= f(u)\quad \text{in}\ \Omega,
   \end{equation}    
    along with the boundary condition
   \begin{equation*}
    u(x)\to \infty \quad \text{as} \quad \text{dist}(x,\partial \Omega)\to 0.
\end{equation*}
The assumption on $f$ is as follows:
\begin{enumerate}
    \myitem{$\textbf{(A4)}$}\label{A4} \textit{$f:[0,\infty)\rightarrow [0,\infty)$ be a continuous, increasing function that is positive on $(0,\infty)$, $f(0)=0$ and satisfies Keller-Osserman condition \ref{A1}.}
\end{enumerate}
For example,  let $Q(r)=r^{p-2}$, if $f(x)=x^{q}$, then $q>p-1$. 

 We consider the notion of a local weak solution as 
 \begin{definition}
    Let $1<q<\infty$,  $q'$ be its conjugate and $\Omega$ be an open set in $\rn$. A function  $u\in\{v\in W^{1,q}_{loc}(\Omega)~|~Q(|\nabla v|)\nabla v\in (L_{loc}^{q'}(\Omega))^{n}\}$ is called a \textit{Weak Solution} to problem \eqref{eqn_noboundary condition} if for any $\Omega'\subset\subset\Omega$
    \begin{align*}
        -\int_{\Omega'}Q(|\nabla u|)\nabla u\cdot \nabla \phi=\int_{\Omega'}f(u)\phi \quad \forall~\phi\in W_{0}^{1,q}(\Omega')
    \end{align*}
    whenever $f(u)\in L_{loc}^{q'}(\Omega)$. And, it is a \textit{weak subsolution} (or, weak supersolution) if
    \begin{align*}
        -\int_{\Omega'}Q(|\nabla u|)\nabla u\cdot\nabla \phi \ dx\geq \int_{\Omega'}f(u)\phi \ dx  \ \  \text{\bigg(or, $\leq  \int_{\Omega'}f(u)\phi \ dx $\bigg)} \ \quad \forall~ \phi\in W_{0}^{1,q}(\Omega'),\phi\geq 0,
    \end{align*} 
    for every $\Omega'\subset \subset \Omega$.
\end{definition}
\begin{remark}
   For the p-Laplacian, the solution $u \in W^{1,p}_{loc}(\Omega)$ as the map $v\longmapsto \int_{\Omega'}|\nabla u|^{p-2}\nabla u\cdot \nabla v$, for a fixed $u$, where $\Omega'\subset\subset\Omega$, is a bounded linear map on $W^{1,p}(\Omega')$.
\end{remark}

One of the most important tools available for dealing with the well-posedness of general nonlinear operators is the comparison principle. We state the same; the proof can be found in \cite[Theorem 2.2]{diaz1993} and \cite[Theorem 2.15]{lindqvist}. 

 \begin{proposition}[\cite{diaz1993}]{(Comparison Principle)}\label{prop_comparison}
    Let $u,v\in W_{loc}^{1,p}(\Omega)\cap C(\Omega)$ be two functions such that 
   \begin{align*}
       div(Q(|\nabla u|)\nabla u- f(u)\geq div(Q(|\nabla v|)\nabla v- f(v)\quad \text{in}~\Omega, \quad \text{weakly}
   \end{align*} 
where $f$ is increasing with $f(0^{+})=0$ and that,
\begin{align*}
    \limsup \frac{u(x)}{v(x)}\leq 1\quad \text{as}~\text{dist}(x,\partial\Omega)\to 0.
\end{align*} 
Then $u\leq v$ in $\Omega$. In particular, if $u$ is a weak subsolution and $v$ is a weak supersolution, the result holds.
    \end{proposition}

\subsection{Local Bounds and Existence:} The next proposition gives a local uniform bound for the solution of \eqref{eqn_noboundary condition}. 

\begin{proposition}\label{prop_weak_local_bound}
    Let $\Omega$ be a bounded domain in $\rn$ and $u\in W_{loc}^{1,p}(\Omega)\cap C(\Omega)$ be a weak subsolution of \eqref{eqn_noboundary condition} with $f$ satisfying \ref{A4}. For any $x_{0}\in \Omega$ and $R>0$ such that $B_{R}(x_{0})\subset \subset \Omega$, we have
    \begin{equation}\label{eqn_local_bound}
        u(x)\leq \omega\Big(\frac{R}{2}\Big) \quad \text{for all } x\in B_{\frac{R}{2}}(x_{0}),
    \end{equation}
    where $\omega$ solves
     \begin{equation}\label{1d-gen-eqn}
        \begin{cases}
            (Q(\omega'(t))\omega'(t))'=f(\omega(t)) \quad &\text{in } (-R,R)\\
            \omega(t)\to\infty \quad &\text{as } t\to\pm R.
        \end{cases}
    \end{equation}
\end{proposition}
 \begin{proof}
     We construct a radially symmetric weak boundary blow-up supersolution $v$ of \eqref{eqn_noboundary condition} on $B_{R}(x_{0})$, and use the comparison principle for $u$ and $v$.

    Define $v(x)=\omega(|x-x_{0}|)\in C(B_{R}(x_{0}))$. For any $\phi\in C_{c}^{\infty}(B_{R}(x_{0}))$, $\phi\geq 0$. By denoting $x=x_{0}+\frac{t}{R}\theta(x)$, where $t\in (0,R)$ and $\theta(x)\in \partial B_{R}(x_{0})$, we get
     \begin{equation}\label{eqn_eta'}
     \begin{split}
    &\int_{B_{R}(x_{0})}Q(|\nabla v |)\nabla v\cdot \nabla\phi \ dx\\
    &=\int_{B_{R}(x_{0})}Q(|\omega'(|x-x_{0}|)|)\omega'(|x-x_{0}|) \frac{x-x_{0}}{|x-x_{0}|}\cdot \nabla \phi \ dx\\
      &=\int_{0}^{R}\int_{\partial B_{t}(x_{0})}Q(|\omega'(t)|)\omega'(t) \theta(x) \cdot \nabla \phi(t,x)  \ dH(x) dt\\
        &=\int_{0}^{R}Q(|\omega'(t)|)\omega'(t)~ \int_{\partial B_{R}(x_{0})}\frac{\partial \phi}{\partial t}(t,x) \frac{t^{n-1}}{R^{n}} \ dH(x) dt.\\ 
         \end{split}
     \end{equation}
    as $\frac{\partial \phi}{\partial t}(t,x)=\theta(x) \cdot \nabla \phi(t,x)$.\\
     
    Define 
    \begin{equation*}
        \eta(t)=\int_{0}^{t} \int_{\partial B_{R}(x_{0})}\frac{\partial \phi}{\partial s}(s,x) \frac{s^{n-1}}{R^{n}} \ dH(x) ds \quad \text{for }t> 0.
    \end{equation*}
    Which is a $C_{c}(\mathbb{R})$ function as $\phi\in C_{c}^{\infty}(B_{R}(x_{0}))$. Then 
    \begin{equation*}
        \eta'(t)=\int_{\partial B_{R}(x_{0})}\frac{\partial \phi}{\partial t}(s,x) \frac{t^{n-1}}{R^{n}} \ dH(x) dt.
    \end{equation*}
    Notice that as $\phi\geq 0$,
    \begin{equation*}
    \begin{split}
        \frac{\partial \phi}{\partial s}(s,x) s^{n-1}&=\frac{\partial}{\partial s}\left(\phi(s,x)s^{n-1}\right)-(n-1)\phi(s,x)s^{n-2}\\
        &\leq \frac{\partial}{\partial s}\left(\phi(s,x)s^{n-1}\right)
        \end{split}
    \end{equation*}
As $\omega$ is a solution of equation \eqref{1d-gen-eqn}, by \eqref{eqn_eta'} we get
    \begin{align*}
        \int_{B_{R}(x_{0})}Q(|\nabla v|)\nabla v\cdot \nabla\phi \ dx&=-\int_{0}^{R}f(\omega(t)) \eta(t) \ dt\\
        &\geq  -\int_{0}^{R}f(\omega(t)) \int_{0}^{t}\int_{\partial B_{R}(x_{0})} \frac{\partial}{\partial s}\left(\phi(s,x)\frac{s^{n-1}}{R^{n}}\right) \ dH(x) dt \\
        &=-\int_{0}^{R}f(\omega(t))\int_{\partial B_{t}(x_{0})}\phi(s,x)\ dH(x) dt \quad \\
        &=-\int_{B_{R}(x_{0})} f(v(x))\phi(x) \ dx.
    \end{align*}
    $v(x)\to \infty$ as $|x|\to R$, therefore $v$ is a large supersolution of \eqref{eqn_noboundary condition}.
Since $u<\infty$ in $B_{R}(x_{0})$, comparison principle (Proposition \ref{prop_comparison}) implies 
    \begin{equation*}
        u(x)\leq v(x)\quad \text{for all }x\in B_{R}(x_{0}).
    \end{equation*}
    Because $\omega$ is increasing in $(0,R)$ we get \eqref{eqn_local_bound}. 
 \end{proof}
The gradient of the sub-solution of \eqref{eqn_noboundary condition} when $Q(r)=r^{p-2}$  is bounded in  $L^{p}_{loc}(\Omega)$, it follows by the usual cutoff technique.
\begin{proposition}\label{prop_grad_lp_bound}
     Let $\Omega$ be a bounded domain in $\rn$, $u\in W_{loc}^{1,p}(\Omega)\cap C(\Omega)$ be any weak sub solution of \eqref{eqn_noboundary condition}  with $Q(r)=r^{p-2}$ and $f:[0,\infty)\rightarrow [0,\infty)$ be a continuous, increasing function that is positive on $(0,\infty)$, $f(0)=0$. Then $\|\nabla u\|_{L^{p}(K)}$ is bounded independent of $u$.
\end{proposition}
\begin{proof}
    Let $K$ be a compact subset of $\Omega$, $K_{1}\subset \Omega$ be compact such that $K\subsetneq K_{1}$ and $\phi\in C_{c}^{\infty}(K_{1})$ such that $\phi\leq 1$, $\phi=1$ on $K$. Given the Proposition \ref{prop_weak_local_bound}, by taking $u\phi$ as a test function in the weak formulation, we get,
     \begin{equation*}
        \begin{split}
            \int_{K_{1}}|\nabla u|^{p}\phi^{2} \ dx &\leq \int_{K_{1}}f(u)u\phi^{2} \ dx -2 \int_{K_{1}}\phi|\nabla u|^{p-2}u \nabla u\cdot \nabla \phi \ dx\\
            &\leq \int_{K_{1}}f(u)u\phi^{2} \ dx + \int_{K_{1}}\phi^{\frac{1}{p}}|\nabla u|^{p-1} \phi^{1-\frac{1}{p}}u  |\nabla \phi| \ dx.\\
        \end{split}
    \end{equation*}
    By the assumption that $f$ is increasing and Proposition \ref{prop_weak_local_bound}, also using young's inequality for $c>0$,
    \begin{equation*}
        \int_{K_{1}}|\nabla u|^{p}\phi \ dx \leq f(\|u\|_{L^{\infty}(K_{1})})\|u\|_{L^{\infty}(K_{1})}|K_{1}|+ \frac{c(p-1)}{p}\int_{K_{1}}\phi |\nabla u|^{p} \ dx + \frac{1}{pc^{p-1}}\int_{K_{1}} \phi^{p-1} u^{p}|\nabla \phi|^{p} \ dx
            \end{equation*}
            Choosing $c=\frac{p}{2(p-1)}$ we get 
            \begin{equation*}
                \int_{K}|\nabla u|^{p} \ dx\leq 2f(\|u\|_{L^{\infty}(K_{1})})\|u\|_{L^{\infty}(K_{1})}|K_{1}|+ \frac{2^{p}(p-1)^{p-1}}{p^{p}}\|\nabla \phi\|^{p}_{L^{p}(K_{1})}\|u\|^{p}_{L^{\infty}(K_{1})}
            \end{equation*}
            The result follows from the previous proposition and the compactness.
\end{proof}
The existence result for a large solution in a bounded domain is stated in the proceeding theorem in a compact form. Proof and related explanations can be found in  \cite[Section 6]{diaz1993} and \cite{materopaper} in the case of the p-Laplace equation.
\begin{theorem}[\cite{materopaper},\cite{diaz1993}]
    Let $\Omega$ be a bounded domain in $\rn$ with $C^{2}$ boundary with $n\geq 2$, let $u_{m}$ solve 
    \begin{align}\label{finite_dirchlet}
        \begin{cases}
          \text{div}\left(Q\left(|\nabla u_{m}|\right)\nabla u_{m}\right)= f(u_{m})  \quad &\text{in} ~\Omega \\
         u_{m}=m \quad &\text{on}~\partial\Omega,
        \end{cases}
        \end{align}
    with $f$ satisfying \ref{A4}.  
    \begin{enumerate}
        \item Then $\{u_{m}\}\subset L_{loc}^{\infty}(\Omega)$ is increasing and bounded.
        \item Let $u(x)=sup~u_{m}(x)$, then $u\in W_{loc}^{1,\infty}(\Omega)$, and $u$ solves 
        \begin{equation}\label{singular_bdry}
          \begin{cases}
              \text{div}\left(Q\left(|\nabla u|\right)\nabla u\right)=f(u)\quad &\text{in } \Omega\\
              u(x)\to \infty \quad &\text{as } \  \text{dist}(x,\partial \Omega)\to 0.
          \end{cases}  
        \end{equation}
    \end{enumerate} 
\end{theorem}
Existence and $C^{1}$ regularity of solution to \eqref{finite_dirchlet} can be found in \cite{diaz1985,DiBenedetto1983}. 

\subsection{A note on Uniqueness:} The uniqueness of boundary blow-up solutions generally requires conditions stronger than the classical Keller-Osserman condition. Such a study is not the primary focus of our article, and for completeness, we discuss the result for $Q(r)=r^{p-2}$.
\begin{enumerate}
    \myitem{$\textbf{(A5)}$}\label{A5}  \textit{Let $f$ be such that 
    \begin{equation*}
       \liminf_{t\to\infty}\frac{\Psi_{p}(\beta t)}{\Psi_{p}(t)}>1\quad \text{for all } \beta\in (0,1) \quad (\Psi_{p} \text{ defined in \ref{A1}}).
    \end{equation*}}
\end{enumerate}
The following result can be found in \cite[Corollary 4.5]{materopaper}; we discuss the proof for completeness.
 \begin{theorem}[\cite{materopaper}]\label{thrm_matero_asym}
    Let $\Omega$ be a bounded domain in $\rn$ with $C^{2}$ boundary and $u\in W_{loc}^{1,p}(\Omega)$ be a solution of \eqref{singular_bdry}  when $Q(r)=r^{p-2}$ , where $f$ satisfy \ref{A4} and \ref{A5}. Then 
    \begin{equation}\label{eqn_matero_asymptotic}
        \displaystyle \lim_{x\to \partial\Omega}\frac{u(x)}{\Phi_{p}(d(x))}=1,
    \end{equation}
    where $\Phi_{p}$ is the inverse of $\Psi_{p}$ and $d(x)=\text{dist}(x,\partial\Omega$).
 \end{theorem}
\begin{proof}
    Since $\Omega$ is $C^{2}$, it has uniform interior and exterior ball conditions. For $z\in\partial\Omega$, let $R>0$ be such that $B^{z}_{R}(x_{0})$ for some $x_{0}\in \Omega$, is the interior ball and, $\Tilde{B}_{R}^{z}(y_{0})$ for some $y_{0}\in \Omega^{c}$, be the exterior ball associated with $z$. For $0<r<R$, we consider two barriers $w_{r}\in W_{loc}^{1,p}(B_{r}^{z}(x_{0}))$ which is radially symmetric, increasing function on $B_{r}^{z}(x_{0})$ and $v_{r}\in W_{loc}^{1,p}(\Tilde{B}^{z}_{r,2R}(y_{0}))$, which is radially symmetric, decreasing function on the annulus $\Tilde{B}^{z}_{r,2R}(y_{0}):=\Tilde{B}^{z}_{2R}(y_{0})\setminus \Tilde{B}_{r}^{z}(y_{0})$. They individually solve 
    \begin{equation*}
        \begin{cases}
            \Delta_{p}w_{r}(t)=f(w_{r}(t)) \quad &\text{in } [0,r)\\
            w_{r}(t)\to \infty \quad &\text{as } t\to r,
        \end{cases}
    \end{equation*}
    and
    \begin{equation*}
    \begin{cases}
      \Delta_{p}v_{r}(t)=f(v_{r}(t)),\quad &\text{for } t\in (r,2R)\\
      v_{r}(t)\to\infty, \quad &\text{as } t\to r\\
      v_{r}(t)\to 0, \quad &\text{as } t\to 2R,
      \end{cases}
    \end{equation*}
    where $f$ has \ref{A4}. $w_{r}$ and $v_{r}$ verifies
    \begin{equation*}
         \displaystyle \lim_{t\to r} \frac{\Psi_{p}(w_{r}(t))}{r-t}= 1.
    \end{equation*}
    and
    \begin{equation*}
        \displaystyle \lim_{t\to r} \frac{\Psi_{p}(v_{r}(t))}{t-r}\leq 1.
    \end{equation*}
    For the existence of such a function, we refer \cite[Corollary 3.5, Proposition 4.2 and Proposition 4.3]{materopaper}. Comparison principle (Proposition \ref{prop_comparison}) implies  $u\leq w_{r}$ on $B_{r}^{z}(x_{0})$  and hence $\Psi_{p}(u)\geq \Psi_{p}(w_{r})$ on $B_{r}^{z}(x_{0})$. Therefore, 
    \begin{equation*}
        \frac{\Psi_{p}(u(x))}{d(x,z)}\frac{t-r}{\Psi_{p}(w_{r}(x))}\geq 1.
    \end{equation*}
    Letting $r\to R$, as $d(x)\leq d(x,z)$
    \begin{equation}\label{eqn_matero inequality}
        \displaystyle \lim_{x\to z}\frac{\Psi_{p}(u(x))}{d(x)}\geq 1.
    \end{equation}
    For the other way inequality, we compare $u$ and $v_{r}$ on $\Tilde{B}^{z}_{r,2R}(y_{0})\cap \Omega$ and letting $r\to R$ and \eqref{eqn_matero inequality},
    \begin{equation*}
        \displaystyle \lim_{x\to z}\frac{\Psi_{p}(u(x))}{d(x)}= 1.
    \end{equation*}
    By the condition \ref{A5} on $f$, one can derive \eqref{eqn_matero_asymptotic}.
\end{proof}
This leads to the uniqueness,
\begin{theorem}\label{thrm_uniqueness_bdd_domain}
    Let $\Omega$, $f$ be as in Theorem \ref{thrm_matero_asym}, then \eqref{singular_bdry}  when $Q(r)=r^{p-2}$ admits a unique solution.
\end{theorem}
\begin{proof}
    If $u$ and $v$ are two solutions, then by \eqref{eqn_matero_asymptotic}
    \begin{equation*}
        \displaystyle \lim_{x\to\partial\Omega} \frac{u(x)}{v(x)}=1.
    \end{equation*}
    Using the comparison principle (Proposition \ref{prop_comparison}), we can derive that $u=v$.
\end{proof}

 \section{Asymptotics on Infinite Cylinder for p-Laplace equation}\label{section 4}
 
In this section, we address the existence and uniqueness of large solutions on infinite cylinders  for the case $Q(r)=r^{p-2}$, $p\geq 2$. We give a different proof in this special case, which avoids monotonicity methods .  

We would like to get the existence by the approximation of solutions on the finite cylinders $B^{n-1}_{1}(0)\times (-\ell,\ell)$. However, to work with a $C^{2}$ bounded domain, we modify $B^{n-1}_{1}(0)\times (-\ell,\ell)$  by attaching suitable domains on either side. Let
 \begin{align*}
     S_{\ell}:=\{x\in\rn:\sum_{1}^{n-1}x_{i}^{2}<1, -\ell<x_{n}<\ell\}\cup\{x\in\rn:\sum_{1}^{n-1}x_{i}^{4}+(x_{n}\pm \ell)^{4}<1\}.
 \end{align*}
whose boundary is at least $C^{2}$.
We note $S_{\ell}\to S$ when $\ell\to\infty$, where $S=(B^{n-1}_{1}(0)\times \mathbb{R}) \subset \rn$.\\
By the comparison principle and the boundary behaviour, we have the following result
\begin{proposition} Let $f$ satisfy \ref{A4} and $u_{\ell}$ solve \eqref{singular_bdry} for $Q(r)=r^{p-2}$ on $S_{\ell}$ in the weak sense. Then,
    \begin{itemize}
    \item[(\textit{i})] $u_{\ell}\geq 0~\forall~\ell>0$ \quad (as $f(0)=0$)
    \item[(\textit{ii})] $\{u_{\ell}(x)\}_{\ell}$ is a decreasing sequence in the following sense:\\
    For $\ell(x)>0$ with $x\in S_{\ell(x)}$, 
    $\{u_{\ell}(x)\}_{\ell\geq\ell(x)}$ is decreasing in $\ell$.
\end{itemize}
\end{proposition}
Define \begin{equation}\label{eqn_dfn_of_u}
    u(x)=\inf_{\ell\geq\ell(x)}u_{\ell}(x) \quad \forall x\in S.
\end{equation}
 Then by Proposition \ref{prop_weak_local_bound}, $u\in L^{\infty}(\Tilde{S})$ for every $\Tilde{S}$ compactly contained in $S$ ($\Tilde{S}\subset\subset S$). Thus,  \begin{equation}\label{eqn_lp convergence}
    u_{\ell}\to u \quad \text{in}~ L^{q}(\Tilde{S})\quad \text{for } 1\leq q\leq \infty  \quad \text{and for any} \quad \Tilde{S}\subset\subset S.
\end{equation}

Our aim is to find a large solution to the following problem 
\begin{equation}\label{eqn_large solution for plaplacianon}
    \begin{cases}
        \Delta_{p}v= f(v)\quad &\text{in}\quad S\\
        v(x)\to \infty \quad &\text{as}~ \ \text{dist}(x,\partial S)\to 0.
    \end{cases}
\end{equation}
The function $u$ defined in (4.1) is a candidate, but one needs to pass the limit through the weak formulation. 
\begin{theorem}
    For $u_{\ell}$,$u$ and $\Tilde{S}$ as above and $f$ satisfying \ref{A4}, we have
    \begin{align*}
        u_{\ell}\rightharpoonup u \quad in~W^{1,p}(\Tilde{S}).
    \end{align*}
   Thus $u\in W_{loc}^{1,p}(S)$.
\end{theorem}

\begin{proof}
  The result follows directly follows from Proposition \ref{prop_weak_local_bound} and  \ref{prop_grad_lp_bound}.
\end{proof} 

 This brings us to the crucial part:
\begin{theorem}\label{thrm_large solution existence}
   For $f$ satisfying \ref{A4}, $p\geq 2$, $u$ defined in \eqref{eqn_dfn_of_u} solves   \eqref{eqn_large solution for plaplacianon} in the weak sense.
\end{theorem}
\begin{proof}
Note that $u(x) \to \infty$ as $x\to \partial S$ by the definition of $u$. 

    We have that $ \{u_{\ell}\}$, for large $\ell$ is uniformly bounded on $\Tilde{S}$ by the Proposition \ref{prop_weak_local_bound} and hence $\{f(u_{\ell})\}$ is bounded uniformly on $\Tilde{S}$. For any $\phi\in C_{c}^{\infty}(S)$, by taking $\Tilde{S}$ such that $supp(\phi)\subset\Tilde{S}$, dominated convergence theorem implies that 
   \begin{align*}
       \int_{\Tilde{S}}f(u_{\ell})\phi\to \int_{\Tilde{S}}f(u)\phi.
   \end{align*}
    That is \begin{equation}\label{eqn_int convergence}
        \int_{\Tilde{S}}|\nabla u_{\ell}|^{p-2}\nabla u_{\ell}\cdot \nabla\phi\to  -\int_{\Tilde{S}}f(u)\phi.
    \end{equation}
    We aim to achieve 
    \begin{equation}\label{eqn_weak convergence of operator}
         \int_{\Tilde{S}}|\nabla u_{\ell}|^{p-2}\nabla u_{\ell}\cdot \nabla\phi\to  \int_{\Tilde{S}}|\nabla u|^{p-2}\nabla u\cdot \nabla\phi,
    \end{equation}
    by showing that $u_{\ell}\to u$ in $W^{1,p}(\Tilde{S})$.
    
    We use a vector-valued inequality 
    (see \cite{lindqvist}): For $x,y\in \rn$ there is a positive constant $c$ such that
     \begin{equation}\label{vector inequlity}
          |x-y|^{p}\leq c ~ (|x|^{p-2}x-|y|^{p-2}y) \cdot (x-y).
     \end{equation}
      Choose a cutoff function $\psi\in C_{c}^{\infty}(\Tilde{S}_{1})$ such that $\psi=1$ on $\Tilde{S}$ and $0\leq \psi \leq 1$, then 
    \begin{equation}\label{eqn_strong convergence}
    \begin{split}
    \int_{\Tilde{S}}|\nabla u_{\ell}-\nabla u|^{p} &\leq \int_{\Tilde{S}_{1}}\psi|\nabla u_{\ell}-\nabla u|^{p}\\
        &\leq c \int_{\Tilde{S}_{1}}\psi \{(|\nabla u_{\ell}|^{p-2}\nabla u_{\ell}-|\nabla u|^{p-2}\nabla u) \cdot (\nabla u_{\ell}-\nabla u)\}.
         \end{split}
    \end{equation}
     Since $\phi\mapsto \int_{\Tilde{S}_{1}}\psi\{|\nabla u|^{p-2}\nabla u \cdot \nabla \phi\}$ is a continuous linear functional on $W^{1,p}(\Tilde{S_{1}})$, by weak convergence
     \begin{align*}
         \int_{\Tilde{S}_{1}}\psi |\nabla u|^{p-2}\nabla u \cdot \nabla (u_{\ell}-u)\to 0\quad \text{as} \quad \ell\to\infty.
     \end{align*}
       Consider,
    \begin{align*}
        \int_{\Tilde{S}_{1}}|\nabla u_{\ell}|^{p-2}\nabla u_{\ell} \cdot \psi \nabla (u_{\ell}-u) &=\int_{\Tilde{S}_{1}}|\nabla u_{\ell}|^{p-2}\nabla u_{\ell} \cdot \{\nabla(\psi(u_\ell-u))-(u_{\ell}-u)\nabla \psi\}\\
        &=I_{1}-I_{2}.
    \end{align*}
    As $\|\nabla u_{\ell}\|_{L^{p}}$ is uniformly bounded with respect to $\ell$ by the Proposition \ref{prop_grad_lp_bound},  there is a $K>0$ such that
    \begin{align*}
        |I_{2}|&\leq \int_{\Tilde{S}_{1}}|\nabla u_{\ell}|^{p-1}|\nabla \psi\|u_{\ell}-u|\leq  \sup_{\Tilde{S}_{1}}|\nabla\psi| \|\nabla u_{\ell}\|_{L^{p}}^{p-1} \|u_{\ell}-u\|_{L^{p}} \leq K \|u_{\ell}-u\|_{L^{p}} \xrightarrow{\ell \to 0}  0.
    \end{align*}
    Taking care of $I_{1}$ is more delicate, we rewrite it as,
     \begin{align*}
        I_{1}
        =\int_{\Tilde{S}_{1}} \big(|\nabla u_{\ell}|^{p-2}\nabla u_{\ell} \cdot \nabla (\psi(u_{\ell}-u))- f(u)(u_{\ell}-u)\psi \big)  +  \int_{\Tilde{S}_{1}} f(u)(u_{\ell}-u)\psi .
    \end{align*}
    The second integral on the right-hand side tends to $0$ when $\ell\to 0$ by \eqref{eqn_lp convergence}.
By \eqref{eqn_int convergence} and density
\begin{align*}
    \int_{\Tilde{S}_{1}}|\nabla u_{\ell}|^{p-2}\nabla u_{\ell}\cdot \nabla\phi\to -\int_{\Tilde{S}_{1}}f(u)\phi, \quad \forall~\phi\in W_{0}^{1,p}(\Tilde{S}_{1}),
\end{align*}
    this implies 
    \begin{align*}
        \biggl|\int_{\Tilde{S}_{1}}|\nabla u_{\ell}|^{p-2}\nabla u_{\ell}\cdot \nabla\phi- f(u)\phi\biggl|\leq \mathcal{O}_{\ell}(1)\|\phi\|_{W_{0}^{1,p}}\quad\forall~\phi\in W_{0}^{1,p}(\Tilde{S}_{1}).
    \end{align*}
    Taking $\phi=(u_{\ell}-u)\psi$, we get,
    \begin{align*}
        \biggl|\int_{\Tilde{S}_{1}}|\nabla u_{\ell}|^{p-2}\nabla u_{\ell}\cdot \nabla\{(u_{\ell}-u)\psi\}- f(u)(u_{\ell}-u)\psi\biggl|&\leq \mathcal{O}_{\ell}(1)\|(u_{\ell}-u)\psi\|_{W_{0}^{1,p}(\Tilde{S_{1}})} \leq \mathcal{O}_{\ell}(1) M.
    \end{align*}
    for some $M>0$, as $u_{\ell},u$ are bounded in $W^{1,p}(\Tilde{S}_{1})$. Thus 
    \begin{equation*}
        \int_{\Tilde{S}_{1}}|\nabla u_{\ell}|^{p-2}\nabla u_{\ell}\cdot \nabla((u_{\ell}-u)\psi)- f(u)(u_{\ell}-u)\psi \to 0 \quad \text{as}~\ell\to \infty.
    \end{equation*}
    Thus $I_1 \to 0$ when $\ell \to 0$. Then $u_{\ell}\to u$ in $W^{1,p}(\Tilde{S})$ follows by \eqref{eqn_strong convergence}.
   
     Up to a subsequence, the dominated convergence theorem \eqref{eqn_weak convergence of operator} holds.
    The proof is complete by \eqref{eqn_int convergence} and \eqref{eqn_weak convergence of operator}.
\end{proof}

 Finally, we show that the large solution of \eqref{eqn_large solution for plaplacianon} coincides with the large solution of the same problem on $B^{n-1}_{1}(0) \subset \mathbb{R}^{n-1}$, as stated in the next theorem.
\begin{theorem}\label{thrm_large solution on S}
    Assume \ref{A4} and $u$ is as in the Theorem \ref{thrm_large solution existence}, then $u$ is independent of the $n^{\text{th}}$ variable. i,e,. $u(x)=u(x')$ where $x=(x',x_{n})$. Moreover, if we denote $v(x')=u(x',x_{n})$ for a fixed $x_{n}$, $v$ solves 
    \begin{equation}\label{Eqn_crossection}
    \begin{cases}
      \Delta_{p}v- f(v)=0 \quad &\text{in}~B^{n-1}_{1}(0)\\
      v\to \infty \quad &\text{as}~ \text{dist}(x,\partial B^{n-1}_{1}(0)) \to 0.
    \end{cases}
    \end{equation}

\end{theorem}
For a proof, we refer to the general case in Theorem \ref{thrm_large solution on S for Q}
    
    \smallskip

\begin{theorem}(\textbf{Uniqueness})
Assume $f$ satisfies \ref{A4} and \ref{A5}, then the equation \eqref{eqn_large solution for plaplacianon} has a unique solution.  
\end{theorem}
\begin{proof}
    The result follows from Theorem \ref{thrm_uniqueness_bdd_domain} and Theorem \ref{thrm_large solution on S}. 
\end{proof}

\section{Asymptotics on infinite cylindrical domains for \eqref{eqn_mainQ}}\label{section 5}

 This section aims to show the existence of boundary blow-up solutions to the general quasilinear problem \eqref{eqn_mainQ}.  We first consider a general condition on $Q$:
 \begin{enumerate}
     \myitem{$\textbf{(A6)}$} \label{A6} \textit{Let $A(r)=Q(r)r$, we assume
     \begin{equation*}
        \frac{A\left\{B^{-1}\left(rf(r)\right)\right\}}{f(r)}\to 0 \quad \text{as~} r\to \infty , 
     \end{equation*}
     where $B$ is defined in \eqref{dfn F and B}.}
 \end{enumerate}
 When $Q(r)=r^{p-2}$, the assumption simplifies to $\lim_{r\to \infty}\frac{r^{p-1}}{f(r)}=0$, which follows from \ref{A1}.

 In the next theorem we discuss the existence result.

\begin{theorem}\label{maintheoremforQ}
    Assume \ref{A4} and \ref{A6}. Let $u_{\ell}$ be solution of \eqref{singular_bdry} on $\Omega=S_{\ell}$, then there exists $u\in W^{1,p}_{loc}(S)$ such that up to a subsequence $u_{\ell}\rightharpoonup u$ in $W^{1,p}_{loc}(S)$. Moreover $u$ solves \eqref{eqn_mainQ}.
\end{theorem}
\begin{proof}
    We first show that the sequence $\{u_{\ell}\}$ is decreasing and that $u$ is the infimum. Further, for $\tilde{S}\subset\subset S$ the operator $v\mapsto -div(Q(|\nabla v|)\nabla v$ from $D\subset W^{1,p}(\tilde{S})$ to $W^{1,p}(\tilde{S})'$ is bounded and $C^{1}_{c}(\tilde{S})\subset D$. Up to subsequence $-div(Q(|\nabla u_{\ell}|)\nabla u_{\ell})\rightharpoonup \mathcal{L}\in W^{1,p}(\tilde{S})'$. By the weak formulation, $\mathcal{L}$ must agree with the right-hand side limit as $\ell\to\infty$. To show that $\mathcal{L}$ is the same as $-div(Q(|\nabla u|)\nabla u)$, we use the monotonicity and hemicontinuity \cite{Zeidler1990} of the operator.

   By the comparison principle (Proposition \ref{prop_comparison}) and \cite[Remark 3.4]{diaz1993}, the sequence $\{u_{\ell}\}_{\ell}$ is decreasing, and let us define 
    \begin{equation*}
        u(x):=\displaystyle\inf_{\ell} u_{\ell}(x).
    \end{equation*}
    Thus, as a byproduct of Proposition \ref{prop_weak_local_bound} we have $u\in L^{\infty}_{loc}(S)$. Then, by dominated convergence theorem $u_{\ell}\to u$ in $L^{p}(\tilde{S})$ .
    Again by Proposition \ref{prop_weak_local_bound} and the local uniform bound of $\{\nabla u_{\ell}\}$ provided by \cite[Corollary 5.5]{diaz1993}, we find that $\{u_{\ell}\}$ is bounded in $W^{1,p}(\tilde{S})$, we denote the bound by `$M$'.  Hence, upto a subsequence, namely  $\{u_{\ell'}\}$, $u_{\ell'}\rightharpoonup u$ in $W^{1,p}(\tilde{S})$.
    
    Further, using $\big(Q(r)\,r\big)'>0$ and the H\"older inequality we get that for any  $v \in W^{1,p}(\tilde{S})$
    \begin{equation*}
    \begin{split}
         \int_{\tilde{S}}Q(|\nabla u_{\ell'}|)\nabla u_{\ell'} \cdot \nabla v \ dx & \leq \int_{\tilde{S}}Q(|\nabla u_{\ell'}|)|\nabla u_{\ell'}| |\nabla v| \ dx \\
         &\leq Q(M)M |\tilde{S}|^{1-1/p}||\nabla v||_{L^{p}(\tilde{S})}.
    \end{split}
    \end{equation*}
    Thus, there exists a subsequence $\{u_{\ell"}\}_{\ell"}$ of $\{u_{\ell'}\}_{\ell'}$ and $\mathcal{L}\in W^{1,p}(\tilde{S})'$ such that 
    \begin{equation}\label{eqn weak convergence in dual}
        \int_{\tilde{S}} Q(|\nabla u_{\ell"}|)\nabla u_{\ell"} \cdot \nabla v \ dx \to \left<\mathcal{L}, v\right> \quad \text{for all } v\in  W^{1,p}(\tilde{S}).
    \end{equation}
    Moreover since for any $\phi\in W_{0}^{1,p}(\tilde{S})$
    \begin{equation*}
        \int_{\tilde{S}} Q(|\nabla u_{\ell"}|)\nabla u_{\ell"} \cdot \nabla \phi \ dx= -\int_{\tilde{S}}f(u_{\ell"})\phi \ dx \to -\int_{\tilde{S}}f(u)\phi \ dx,
    \end{equation*}
    as $\ell"\to 0$, we get that 
    \begin{equation}\label{L equals to f}
        \left<\mathcal{L}, \phi\right>=-\int_{\tilde{S}}f(u)\phi \ dx \quad \text{for all } \phi\in W_{0}^{1,p}(\tilde{S}).
    \end{equation}
    
    Next, we aim to show 
    \begin{equation*}
        \left<\mathcal{L}, \phi\right>=\int_{\tilde{S}} Q(|\nabla u|)\nabla u \cdot \nabla \phi \ dx \quad \text{for all } \phi \in W_{0}^{1,p}(\tilde{S}).
    \end{equation*}
    To achieve it, we use monotonicity or Browder-Minty argument (c.f. \cite{Zeidler1990,evans}).
    
    By monotonicity \cite[Lemma 2.1]{diaz1993}, we get
    \begin{equation*}
        \int_{\tilde{S}} \left(Q(|\nabla u_{\ell"}|)\nabla u_{\ell"}-Q(|\nabla v|)\nabla v\right)\cdot \nabla(u_{\ell"}-v) \ dx \geq 0, \quad \text{for any } v \in W^{1,p}(\tilde{S}).
   \end{equation*}
   Further by $u_{\ell"}\rightharpoonup u$ and \eqref{eqn weak convergence in dual} we derive by taking $\ell"\to \infty$, that for any $v\in W^{1,p}(\tilde{S})$ 
   \begin{equation*}
       \left<\mathcal{L}-\int_{\tilde{S}}Q(|\nabla v|)\nabla v, u-v\right>\geq 0.
   \end{equation*}
   Now for any $w\in W^{1,p}(\tilde{S})$ substitute $v=u-tw$, $t\in [0,1]$ to get 
   \begin{equation}\label{BM}
       \left<\mathcal{L}-\int_{\tilde{S}}Q(|\nabla (u-tw)|)\nabla (u-tw), w\right>\geq 0.
   \end{equation}
   As the map $t\longmapsto \int_{\tilde{S}}Q(|\nabla (u-tw)|)\nabla (u-tw)\cdot \nabla w$ is continuous for any $w\in W^{1,p}(\tilde{S})$, $t\to 0$ in \eqref{BM} implies 
   \begin{equation*}
       \left<\mathcal{L}-\int_{\tilde{S}}Q(|\nabla u|)\nabla u, w\right>\geq 0.
   \end{equation*}
   Replacing $w$ with $-w$ in the above inequality together with \eqref{L equals to f} we get
   \begin{equation}
       \int_{\tilde{S}}Q(|\nabla u|)\nabla u \cdot \nabla \phi \ dx = -\int_{\tilde{S}}f(u)\phi \ dx \quad \text{for all } \phi\in W_{0}^{1,p}(\tilde{S}).
   \end{equation}
\end{proof}
\begin{remark}
    The first half of the argument in the above theorem gives that for $1<q<\infty$ there exists a subsequence depending on $q$ of $\{u_{\ell}\}$ such that $u_{\ell'(q)}\rightharpoonup u$ in $W^{1,q}(\tilde{S})$.
\end{remark}

This solution is translation invariant with respect to the unbounded variable as given in the next theorem.
\begin{theorem}\label{thrm_large solution on S for Q}
    Assume \ref{A4} and let $u$ be a solution of \eqref{eqn_mainQ}, then it is independent of the $n^{\text{th}}$ variable. i,e,. $u(x)=u(x')$ where $x=(x',x_{n})$. Moreover, if we denote $v(x')=u(x',x_{n})$ for a fixed $x_{n}$, $v$ solves 
    \begin{equation}\label{Eqn_crossection}
    \begin{cases}
      div(Q(|\nabla v|)\nabla v)= f(v) \quad &\text{in}~B^{n-1}_{1}(0)\\
       v\to \infty \quad &  \text{as}~\text{dist}(x,\partial B^{n-1}_{1}(0))\to 0.
    \end{cases}
    \end{equation}
\end{theorem}
\begin{proof}
    Let $\tau=(0,\cdots,0,\tau_{n})$ with $\tau_{n}\in \mathbb{R}$, define 
    \begin{equation*}
        u_{\tau}(x)=u(x+\tau).
    \end{equation*} 
     As $S= B_{1}^{n-1}(0) \times \mathbb{R}$ and $u$ is a weak solution of \eqref{eqn_mainQ}, for any $\psi \in C_c^{\infty}(\tilde{S})$, $\tilde{S}\subset\subset S$. Denoting $\tilde{S}_{\tau}:=\{x+\tau~|~x\in \tilde{S}\}$ we have by the change of variable
    \begin{align*}
      & \int_{\tilde{S}}Q(|\nabla u_{\tau}(x)|)\nabla u_{\tau}(x)\cdot \nabla\psi(x)\ dx
      =\int_{S}Q(|\nabla u(x)|)\nabla u(x)\cdot \nabla\psi(x-\tau)\ dx\\
      &=- \int_{S}f(u(x))\psi(x)\ dx=- \int_{S}f(u(x+\tau))\psi(x)\ dx=- \int_{\tilde{S}}f(u_{\tau}(x))\psi(x)\ dx.
    \end{align*}
    Thus $u_{\tau}$ also solves \eqref{eqn_mainQ}. Then by comparison principle (Proposition \ref{prop_comparison}) $u(x+\tau)\leq u(x)$. Replacing $\tau$ by $-\tau$ we get $u(x-\tau)\leq u(x)$.
    Then for $y=x+\tau$ we have, $u(x)=u(y-\tau)\leq u(y)=u(x+\tau)$. Thus $u(x)$ is independent of the last variable.
    \smallskip
    To see $v$ is the large solution on $B_{1}^{n-1}(0)$, consider $\phi\in C_{c}^{\infty}(B^{n-1}_{1}(0))$, take $\Tilde{\phi}\in C_{c}^{\infty} (\re)$ such that $\int_{\mathbb{R}} \Tilde{\phi} =1$ and define $\psi(x)=\phi(x')\Tilde{\phi}(x_{n})$, Clearly $\psi\in C_{c}^{\infty}(S)$. Then,
    \begin{equation}\label{thrm_test_fn_v}
        -\int_{S}Q(|\nabla u(x)|)\nabla u(x)\cdot \big(\Tilde{\phi}(x_n)\, \nabla_{x'}\phi (x'),\phi(x') \, \Tilde{\phi}'(x_n)\big)\ dx=\int_{S}f(u(x))\psi(x)\ dx,
    \end{equation}
    where $\nabla u=(\nabla_{x'}v,0)$. Now separating the domain of integration in $x'$ and $x_n$ variables, and using  $\int_{\mathbb{R}} \Tilde{\phi} =1$ we find
        \begin{align*}
        -\int_{B^{n-1}_{1}(0)}Q(|\nabla_{x'}v(x')|)\nabla_{x'}v(x')\cdot \nabla_{x'}\phi(x')\ dx' = \int_{B^{n-1}_{1}(0)}f(v(x')\phi(x')\ dx'.
    \end{align*}
    This completes the proof.
\end{proof}

\subsection{Extensions:}\label{subsection-ext}

 Here we discuss the generalisation of the results for domains that become unbounded in more than one direction but not all, and having $C^{2}$ cross-section.

We consider the infinite cylindrical domain  $S(\omega)=\omega \times \re^{n-m}$, where $1\leq m<n$ and $\omega\subset \re^{m}$ is an open bounded domain with $C^{2}$ boundary and take the following equation  \begin{equation}\label{singular_bdry_general_dom}
          \begin{cases}
              \text{div}\left(Q\left(|\nabla u|\right)\nabla u\right)=f(u)\quad &\text{in } S(\omega)\\
              u(x)\to \infty \quad &\text{as } \text{dist}(x,\partial S(\omega))\to 0.
          \end{cases}  
        \end{equation}
Let $S_{\ell}(\omega)$ be a $C^{2}$ bounded domain such that $\omega\times (-\ell,\ell)^{n-m}\subset S_{\ell}(\omega)$, and $S_{\ell}(\omega)\to S(\omega)$ as $\ell\to \infty$.
\begin{theorem}
    Assume \ref{A4} and \ref{A6}. Let $w_{\ell}$ be a solution of \eqref{singular_bdry} on $S_{\ell}(\omega)$,
   define $\displaystyle w(x)=\lim_{\ell\to\infty}w_{\ell}(x)$, for $x\in S(\omega)$, then
\begin{itemize}
    \item[\textit{(i)}] The function $w\in W^{1,p}_{loc}(S(\omega))$ and $w_{\ell}\rightharpoonup w$ in $W^{1,p}_{loc}(S(\omega))$.
    \item[\textit{(ii)}] The function $w$ solves the equation \eqref{singular_bdry_general_dom}.

    \item[\textit{(iii)}] Let $x=(x_{1},x_{2})$, where $x_{1}\in \omega$ and $x_{2}\in \re^{n-m}$. The solution $w$ is independent of $x_{2}$ variable and  for each fixed $x_{2}\in \re^{n-m}$, the function $w(\cdot,x_{2})$ solves 
\begin{equation*}
    \begin{cases}
      div(Q(|\nabla v|)\nabla v)= f(v) \quad &\text{in}~\omega\\
       v\to \infty \quad & \text{as}~\text{dist}(x,\omega)\to 0.
    \end{cases}
    \end{equation*}
 ``Large solution on the infinite cylinder coincides with the large solution on the cross-section $\omega$".
 
\item[\textit{(iv)}]  In addition for $Q(t) = t^{p-2}$, if $f$ satisfies \ref{A5},  then  \eqref{singular_bdry_general_dom} has a unique solution .
\end{itemize}
\end{theorem}

Apart from the usual technical generalisation, the proof is similar to the proofs in Section \ref{section 5}, and we omit them here.

  \begin{remark}
      The large solution on $S(\omega)$ is unique whenever the large solution on $\omega$ is unique.
  \end{remark}
 
 \section*{Acknowledgements}
I.C.  is supported by the DST-INDIA INSPIRE faculty fellowship (IFA22-MA187) and ANRF PMECRG (ANRF/ECRG/2024/004743/PMS). N.N.D. is supported by PMRF grant (2302262).

The authors thank anonymous referees for
helping us improve the paper.

\renewcommand\refname{Bibliography}
\bibliographystyle{abbrv}
\bibliography{ref.bib}

@article{osti_4157678,
 AUTHOR = {Keller, Joseph B.},
     TITLE = {Electrohydrodynamics. {I}. {T}he equilibrium of a charged gas
              in a container},
   JOURNAL = {J. Rational Mech. Anal.},
  FJOURNAL = {Journal of Rational Mechanics and Analysis},
    VOLUME = {5},
      YEAR = {1956},
     PAGES = {715--724},
      ISSN = {1943-5282,1943-5290},
   MRCLASS = {78.0X},
  MRNUMBER = {81748},
MRREVIEWER = {C.\ H.\ Papas},
       DOI = {10.1512/iumj.1956.5.55025},
       URL = {https://doi.org/10.1512/iumj.1956.5.55025},
}

@article{bandle2013large,
 AUTHOR = {Bandle, Catherine and Chipot, Michel},
     TITLE = {Large solutions in cylindrical domains},
   JOURNAL = {Adv. Math. Sci. Appl.},
  FJOURNAL = {Advances in Mathematical Sciences and Applications},
    VOLUME = {23},
      YEAR = {2013},
    NUMBER = {2},
     PAGES = {461--476},
      ISSN = {1343-4373},
      ISBN = {978-4-7625-0666-6},
   MRCLASS = {35J91 (35A01 35A02 35B44)},
  MRNUMBER = {3236627},
}

@article{Keller1957,
 AUTHOR = {Keller, J. B.},
     TITLE = {On solutions of {$\Delta u=f(u)$}},
   JOURNAL = {Comm. Pure Appl. Math.},
  FJOURNAL = {Communications on Pure and Applied Mathematics},
    VOLUME = {10},
      YEAR = {1957},
     PAGES = {503--510},
      ISSN = {0010-3640,1097-0312},
   MRCLASS = {35.0X},
  MRNUMBER = {91407},
MRREVIEWER = {F.\ W.\ Perkins},
       DOI = {10.1002/cpa.3160100402},
       URL = {https://doi.org/10.1002/cpa.3160100402},
}

@article{Osserman1957,
AUTHOR = {Osserman, Robert},
     TITLE = {On the inequality {$\Delta u\geq f(u)$}},
   JOURNAL = {Pacific J. Math.},
  FJOURNAL = {Pacific Journal of Mathematics},
    VOLUME = {7},
      YEAR = {1957},
     PAGES = {1641--1647},
      ISSN = {0030-8730,1945-5844},
   MRCLASS = {35.00},
  MRNUMBER = {98239},
MRREVIEWER = {R.\ M.\ Redheffer},
       URL = {http://projecteuclid.org/euclid.pjm/1103043236},
}

@article{diaz1993,
AUTHOR = {D\'{\i}az, G. and Letelier, R.},
     TITLE = {Explosive solutions of quasilinear elliptic equations:
              existence and uniqueness},
   JOURNAL = {Nonlinear Anal.},
  FJOURNAL = {Nonlinear Analysis. Theory, Methods \& Applications. An
              International Multidisciplinary Journal},
    VOLUME = {20},
      YEAR = {1993},
    NUMBER = {2},
     PAGES = {97--125},
      ISSN = {0362-546X,1873-5215},
   MRCLASS = {35B40 (35J60)},
  MRNUMBER = {1200384},
MRREVIEWER = {V.\ Raghavendra},
       DOI = {10.1016/0362-546X(93)90012-H},
       URL = {https://doi.org/10.1016/0362-546X(93)90012-H},
}

@book{diaz1985,
AUTHOR = {D\'{\i}az, J. I.},
     TITLE = {Nonlinear partial differential equations and free boundaries.
              {V}ol. {I}},
    SERIES = {Research Notes in Mathematics},
    VOLUME = {106},
      NOTE = {Elliptic equations},
 PUBLISHER = {Pitman (Advanced Publishing Program), Boston, MA},
      YEAR = {1985},
     PAGES = {vii+323},
      ISBN = {0-273-08572-7},
   MRCLASS = {35J35 (35G20)},
  MRNUMBER = {853732},
MRREVIEWER = {G.\ Cimatti},
}

@book{lindqvist,
 AUTHOR = {Lindqvist, Peter},
     TITLE = {Notes on the stationary {$p$}-{L}aplace equation},
    SERIES = {SpringerBriefs in Mathematics},
 PUBLISHER = {Springer, Cham},
      YEAR = {2019},
     PAGES = {xi+104},
      ISBN = {978-3-030-14500-2; 978-3-030-14501-9},
   MRCLASS = {35-02 (31C45 35A01 35B65 35D30 35D40 35J60 35J92)},
  MRNUMBER = {3931688},
MRREVIEWER = {Vladimir\ Bobkov},
       DOI = {10.1007/978-3-030-14501-9},
       URL = {https://doi.org/10.1007/978-3-030-14501-9},
}

@article{Bandle1998OnSE,
  AUTHOR = {Bandle, Catherine and Marcus, Moshe},
     TITLE = {On second-order effects in the boundary behaviour of large
              solutions of semilinear elliptic problems},
   JOURNAL = {Differential Integral Equations},
  FJOURNAL = {Differential and Integral Equations. An International Journal
              for Theory \& Applications},
    VOLUME = {11},
      YEAR = {1998},
    NUMBER = {1},
     PAGES = {23--34},
      ISSN = {0893-4983},
   MRCLASS = {35J60 (34C99 35B40)},
  MRNUMBER = {1607972},
}

@article{Bandlemarcus1,
AUTHOR = {Bandle, Catherine and Marcus, Moshe},
     TITLE = {Dependence of blowup rate of large solutions of semilinear
              elliptic equations, on the curvature of the boundary},
   JOURNAL = {Complex Var. Theory Appl.},
  FJOURNAL = {Complex Variables. Theory and Application. An International
              Journal},
    VOLUME = {49},
      YEAR = {2004},
    NUMBER = {7-9},
     PAGES = {555--570},
      ISSN = {0278-1077,1563-5066},
   MRCLASS = {35J60 (35B45 35J25)},
  MRNUMBER = {2088047},
MRREVIEWER = {Lubomira\ G.\ Softova},
       DOI = {10.1080/02781070410001731729},
       URL = {https://doi.org/10.1080/02781070410001731729},
}

@article{lions,
  AUTHOR = {Lasry, J.-M. and Lions, P.-L.},
     TITLE = {Nonlinear elliptic equations with singular boundary conditions
              and stochastic control with state constraints. {I}. {T}he
              model problem},
   JOURNAL = {Math. Ann.},
  FJOURNAL = {Mathematische Annalen},
    VOLUME = {283},
      YEAR = {1989},
    NUMBER = {4},
     PAGES = {583--630},
      ISSN = {0025-5831,1432-1807},
   MRCLASS = {35J60 (49A60 93E20)},
  MRNUMBER = {990591},
MRREVIEWER = {Jose\ Luis\ Menaldi},
       DOI = {10.1007/BF01442856},
       URL = {https://doi.org/10.1007/BF01442856},
}

@article{Lieberman,
 AUTHOR = {Lieberman, Gary M.},
     TITLE = {Asymptotic behavior and uniqueness of blow-up solutions of
              quasilinear elliptic equations},
   JOURNAL = {J. Anal. Math.},
  FJOURNAL = {Journal d'Analyse Math\'{e}matique},
    VOLUME = {115},
      YEAR = {2011},
     PAGES = {213--249},
      ISSN = {0021-7670,1565-8538},
   MRCLASS = {35J91 (35B40 35B44)},
  MRNUMBER = {2855038},
MRREVIEWER = {Tomasz\ Adamowicz},
       DOI = {10.1007/s11854-011-0028-5},
       URL = {https://doi.org/10.1007/s11854-011-0028-5},
}

@article{Costin,
    AUTHOR = {Costin, O. and Dupaigne, L.},
     TITLE = {Boundary blow-up solutions in the unit ball: asymptotics,
              uniqueness and symmetry},
   JOURNAL = {J. Differential Equations},
  FJOURNAL = {Journal of Differential Equations},
    VOLUME = {249},
      YEAR = {2010},
    NUMBER = {4},
     PAGES = {931--964},
      ISSN = {0022-0396,1090-2732},
   MRCLASS = {35J91 (35B07 35B44 35C20)},
  MRNUMBER = {2652158},
MRREVIEWER = {Dian\ K.\ Palagachev},
       DOI = {10.1016/j.jde.2010.02.023},
       URL = {https://doi.org/10.1016/j.jde.2010.02.023},
}

@article{Payne,
AUTHOR = {Flavin, J. N. and Knops, R. J. and Payne, L. E.},
     TITLE = {Asymptotic behaviour of solutions to semi-linear elliptic
              equations on the half-cylinder},
   JOURNAL = {Z. Angew. Math. Phys.},
  FJOURNAL = {Zeitschrift f\"{u}r Angewandte Mathematik und Physik. ZAMP.
              Journal of Applied Mathematics and Physics. Journal de
              Math\'{e}matiques et de Physique Appliqu\'{e}es},
    VOLUME = {43},
      YEAR = {1992},
    NUMBER = {3},
     PAGES = {405--421},
      ISSN = {0044-2275,1420-9039},
   MRCLASS = {35J65 (35B40)},
  MRNUMBER = {1166964},
MRREVIEWER = {Allan\ L.\ Edelson},
       DOI = {10.1007/BF00946237},
       URL = {https://doi.org/10.1007/BF00946237},
}

@article {MR1240799,
   AUTHOR = {Marcus, Moshe and V\'{e}ron, Laurent},
     TITLE = {Uniqueness of solutions with blowup at the boundary for a
              class of nonlinear elliptic equations},
   JOURNAL = {C. R. Acad. Sci. Paris S\'{e}r. I Math.},
  FJOURNAL = {Comptes Rendus de l'Acad\'{e}mie des Sciences. S\'{e}rie I.
              Math\'{e}matique},
    VOLUME = {317},
      YEAR = {1993},
    NUMBER = {6},
     PAGES = {559--563},
      ISSN = {0764-4442},
   MRCLASS = {35J65 (35B40)},
  MRNUMBER = {1240799},
MRREVIEWER = {Andrzej\ Szulkin},
}

@article{Buccheri,
    AUTHOR = {Buccheri, Stefano and Leonori, Tommaso},
     TITLE = {Large solutions to quasilinear problems involving the
              {$p$}-{L}aplacian as {$p$} diverges},
   JOURNAL = {Calc. Var. Partial Differential Equations},
  FJOURNAL = {Calculus of Variations and Partial Differential Equations},
    VOLUME = {60},
      YEAR = {2021},
    NUMBER = {1},
     PAGES = {Paper No. 30, 23},
      ISSN = {0944-2669,1432-0835},
   MRCLASS = {35J25 (35J92)},
  MRNUMBER = {4201653},
MRREVIEWER = {Lubomira\ G.\ Softova},
       DOI = {10.1007/s00526-020-01883-6},
       URL = {https://doi.org/10.1007/s00526-020-01883-6},
}

@article{bandlemarcus3,
    AUTHOR = {Bandle, Catherine and Marcus, Moshe},
     TITLE = {Asymptotic behaviour of solutions and their derivatives, for
              semilinear elliptic problems with blowup on the boundary},
   JOURNAL = {Ann. Inst. H. Poincar\'{e} C Anal. Non Lin\'{e}aire},
  FJOURNAL = {Annales de l'Institut Henri Poincar\'{e} C. Analyse Non
              Lin\'{e}aire},
    VOLUME = {12},
      YEAR = {1995},
    NUMBER = {2},
     PAGES = {155--171},
      ISSN = {0294-1449,1873-1430},
   MRCLASS = {35J60 (35B40 35J67)},
  MRNUMBER = {1326666},
MRREVIEWER = {Allan\ L.\ Edelson},
       DOI = {10.1016/S0294-1449(16)30162-7},
       URL = {https://doi.org/10.1016/S0294-1449(16)30162-7},
}

@article{bandlemarcus4,
    AUTHOR = {Bandle, Catherine and Marcus, Moshe},
     TITLE = {``{L}arge'' solutions of semilinear elliptic equations:
              existence, uniqueness and asymptotic behaviour},
      NOTE = {Festschrift on the occasion of the 70th birthday of Shmuel
              Agmon},
   JOURNAL = {J. Anal. Math.},
  FJOURNAL = {Journal d'Analyse Math\'{e}matique},
    VOLUME = {58},
      YEAR = {1992},
     PAGES = {9--24},
      ISSN = {0021-7670,1565-8538},
   MRCLASS = {35J65 (35B40)},
  MRNUMBER = {1226934},
MRREVIEWER = {Rolando\ Magnanini},
       DOI = {10.1007/BF02790355},
       URL = {https://doi.org/10.1007/BF02790355},
}

@article{veronmarcus,
    AUTHOR = {Marcus, M. and V\'{e}ron, L.},
     TITLE = {Uniqueness and asymptotic behavior of solutions with boundary
              blow-up for a class of nonlinear elliptic equations},
   JOURNAL = {Ann. Inst. H. Poincar\'{e} C Anal. Non Lin\'{e}aire},
  FJOURNAL = {Annales de l'Institut Henri Poincar\'{e} C. Analyse Non
              Lin\'{e}aire},
    VOLUME = {14},
      YEAR = {1997},
    NUMBER = {2},
     PAGES = {237--274},
      ISSN = {0294-1449,1873-1430},
   MRCLASS = {35J60 (35B40)},
  MRNUMBER = {1441394},
MRREVIEWER = {Georg\ S.\ Weiss},
       DOI = {10.1016/S0294-1449(97)80146-1},
       URL = {https://doi.org/10.1016/S0294-1449(97)80146-1},
}

@article{bandlesurvay,
    AUTHOR = {Bandle, Catherine},
     TITLE = {Asymptotic behavior of large solutions of elliptic equations},
   JOURNAL = {An. Univ. Craiova Ser. Mat. Inform.},
  FJOURNAL = {Analele Universit\u{a}\c{t}ii din Craiova. Seria
              Matematic\u{a}-Informatic\u{a}. Annals of the University of
              Craiova. Mathematics and Computer Science Series},
    VOLUME = {32},
      YEAR = {2005},
     PAGES = {1--8},
      ISSN = {1223-6934,2246-9958},
   MRCLASS = {35J60 (35B40)},
  MRNUMBER = {2215890},
}

@article{cbandle,
     AUTHOR = {Bandle, Catherine},
     TITLE = {Asymptotic behaviour of large solutions of quasilinear
              elliptic problems},
      NOTE = {Special issue dedicated to Lawrence E. Payne},
   JOURNAL = {Z. Angew. Math. Phys.},
  FJOURNAL = {Zeitschrift f\"{u}r Angewandte Mathematik und Physik. ZAMP.
              Journal of Applied Mathematics and Physics. Journal de
              Math\'{e}matiques et de Physique Appliqu\'{e}es},
    VOLUME = {54},
      YEAR = {2003},
    NUMBER = {5},
     PAGES = {731--738},
      ISSN = {0044-2275,1420-9039},
   MRCLASS = {35J60 (35B40)},
  MRNUMBER = {2019176},
MRREVIEWER = {Pierpaolo\ Esposito},
       DOI = {10.1007/s00033-003-3207-0},
       URL = {https://doi.org/10.1007/s00033-003-3207-0},
}

@article{McKenna,
     AUTHOR = {Lazer, A. C. and McKenna, P. J.},
     TITLE = {Asymptotic behavior of solutions of boundary blowup problems},
   JOURNAL = {Differential Integral Equations},
  FJOURNAL = {Differential and Integral Equations. An International Journal
              for Theory and Applications},
    VOLUME = {7},
      YEAR = {1994},
    NUMBER = {3-4},
     PAGES = {1001--1019},
      ISSN = {0893-4983},
   MRCLASS = {35J60 (34B15 35B40)},
  MRNUMBER = {1270115},
MRREVIEWER = {Xiaoyun\ Ma},
}

@article{Pino,
   AUTHOR = {del Pino, Manuel and Letelier, Ren\'{e}},
     TITLE = {The influence of domain geometry in boundary blow-up elliptic
              problems},
   JOURNAL = {Nonlinear Anal.},
  FJOURNAL = {Nonlinear Analysis. Theory, Methods \& Applications. An
              International Multidisciplinary Journal},
    VOLUME = {48},
      YEAR = {2002},
    NUMBER = {6},
     PAGES = {897--904},
      ISSN = {0362-546X,1873-5215},
   MRCLASS = {35J65 (35B40 35J60)},
  MRNUMBER = {1879080},
MRREVIEWER = {Mich\`ele\ Grillot},
       DOI = {10.1016/S0362-546X(00)00222-4},
       URL = {https://doi.org/10.1016/S0362-546X(00)00222-4},
}

@article{kichenassamy,
    AUTHOR = {Kichenassamy, Satyanad},
     TITLE = {R\'{e}gularit\'{e} du rayon hyperbolique},
   JOURNAL = {C. R. Math. Acad. Sci. Paris},
  FJOURNAL = {Comptes Rendus Math\'{e}matique. Acad\'{e}mie des Sciences.
              Paris},
    VOLUME = {338},
      YEAR = {2004},
    NUMBER = {1},
     PAGES = {13--18},
      ISSN = {1631-073X,1778-3569},
   MRCLASS = {35J60 (35B65)},
  MRNUMBER = {2038076},
       DOI = {10.1016/j.crma.2003.10.037},
       URL = {https://doi.org/10.1016/j.crma.2003.10.037},
}

@article{matero,
     AUTHOR = {Matero, J.},
     TITLE = {Boundary--blow up problems in a fractal domain},
   JOURNAL = {Z. Anal. Anwendungen},
  FJOURNAL = {Zeitschrift f\"{u}r Analysis und ihre Anwendungen. Journal for
              Analysis and its Applications},
    VOLUME = {15},
      YEAR = {1996},
    NUMBER = {2},
     PAGES = {419--444},
      ISSN = {0232-2064,1661-4534},
   MRCLASS = {35J65 (35B05 35B40)},
  MRNUMBER = {1394436},
MRREVIEWER = {Rolando\ Magnanini},
       DOI = {10.4171/ZAA/708},
       URL = {https://doi.org/10.4171/ZAA/708},
}

@book{materothisis,
   AUTHOR = {Matero, Jerk Herman},
     TITLE = {Nonlinear elliptic problems with boundary blow-up},
      NOTE = {Thesis (Takn.dr)--Uppsala Universitet (Sweden)},
 PUBLISHER = {ProQuest LLC, Ann Arbor, MI},
      YEAR = {1997},
     PAGES = {181},
      ISBN = {978-9150-61204-2},
   MRCLASS = {99-05},
  MRNUMBER = {2715080},
       URL =
              {http://gateway.proquest.com/openurl?url_ver=Z39.88-2004&rft_val_fmt=info:ofi/fmt:kev:mtx:dissertation&res_dat=xri:pqdiss&rft_dat=xri:pqdiss:C578545},
}

@article{bandlegiarrusso,
AUTHOR = {Bandle, Catherine and Giarrusso, Ester},
     TITLE = {Boundary blow up for semilinear elliptic equations with
              nonlinear gradient terms},
   JOURNAL = {Adv. Differential Equations},
  FJOURNAL = {Advances in Differential Equations},
    VOLUME = {1},
      YEAR = {1996},
    NUMBER = {1},
     PAGES = {133--150},
      ISSN = {1079-9389},
   MRCLASS = {35J60 (35B40)},
  MRNUMBER = {1357958},
}

@article{biberbach,
    AUTHOR = {Bieberbach, Ludwig},
     TITLE = {{$\Delta u=e^u$} und die automorphen {F}unktionen},
   JOURNAL = {Math. Ann.},
  FJOURNAL = {Mathematische Annalen},
    VOLUME = {77},
      YEAR = {1916},
    NUMBER = {2},
     PAGES = {173--212},
      ISSN = {0025-5831,1432-1807},
   MRCLASS = {99-04},
  MRNUMBER = {1511854},
       DOI = {10.1007/BF01456901},
       URL = {https://doi.org/10.1007/BF01456901},
}

@incollection{essen,
 AUTHOR = {Bandle, Catherine and Ess\'{e}n, Matts},
     TITLE = {On the solutions of quasilinear elliptic problems with
              boundary blow-up},
 BOOKTITLE = {Partial differential equations of elliptic type ({C}ortona,
              1992)},
    SERIES = {Sympos. Math.},
    VOLUME = {XXXV},
     PAGES = {93--111},
 PUBLISHER = {Cambridge Univ. Press, Cambridge},
      YEAR = {1994},
      ISBN = {0-521-46048-4},
   MRCLASS = {35J65 (35B05)},
  MRNUMBER = {1297774},
}

@book{evans,
    author ={L. C. Evans} ,
    title ={Partial Differential Equations: Second Edition
} ,
    publisher ={AMS Graduate Series in Mathematics} ,
    year = {2010}
}

@article{materopaper,
AUTHOR = {Matero, Jerk},
     TITLE = {Quasilinear elliptic equations with boundary blow-up},
   JOURNAL = {J. Anal. Math.},
  FJOURNAL = {Journal d'Analyse Math\'{e}matique},
    VOLUME = {69},
      YEAR = {1996},
     PAGES = {229--247},
      ISSN = {0021-7670,1565-8538},
   MRCLASS = {35J65 (35B40 35J70)},
  MRNUMBER = {1428101},
MRREVIEWER = {Dian\ K.\ Palagachev},
       DOI = {10.1007/BF02787108},
       URL = {https://doi.org/10.1007/BF02787108},
}

@article{Zongming2007,
    AUTHOR = {Guo, Zongming and Shang, Junli},
     TITLE = {Remarks on uniqueness of boundary blow-up solutions},
   JOURNAL = {Nonlinear Anal.},
  FJOURNAL = {},
    VOLUME = {66},
      YEAR = {2007},
    NUMBER = {2},
     PAGES = {484--497},
      ISSN = {0362-546X,1873-5215},
   MRCLASS = {35J60 (34B15 35B40)},
  MRNUMBER = {2279541},
       DOI = {10.1016/j.na.2005.11.042},
       URL = {https://doi.org/10.1016/j.na.2005.11.042},
}

@article{Du2003,
    AUTHOR = {Du, Yihong and Guo, Zongming},
     TITLE = {Boundary blow-up solutions and their applications in
              quasilinear elliptic equations},
   JOURNAL = {J. Anal. Math.},
  FJOURNAL = {Journal d'Analyse Math\'{e}matique},
    VOLUME = {89},
      YEAR = {2003},
     PAGES = {277--302},
      ISSN = {0021-7670,1565-8538},
   MRCLASS = {35J60 (35B45 35J67)},
  MRNUMBER = {1981921},
MRREVIEWER = {Dian\ K.\ Palagachev},
       DOI = {10.1007/BF02893084},
       URL = {https://doi.org/10.1007/BF02893084},
}

@article{Amandine1997,
    AUTHOR = {Aftalion, Amandine and Reichel, Wolfgang},
     TITLE = {Existence of two boundary blow-up solutions for semilinear
              elliptic equations},
   JOURNAL = {J. Differential Equations},
  FJOURNAL = {Journal of Differential Equations},
    VOLUME = {141},
      YEAR = {1997},
    NUMBER = {2},
     PAGES = {400--421},
      ISSN = {0022-0396,1090-2732},
   MRCLASS = {35J60 (35B40)},
  MRNUMBER = {1488359},
       DOI = {10.1006/jdeq.1997.3324},
       URL = {https://doi.org/10.1006/jdeq.1997.3324},
}

@article{Shuibo2010,
   AUTHOR = {Huang, Shuibo and Tian, Qiaoyu and Zhang, Shengzhi and Xi,
              Jinhua and Fan, Zheng-en},
     TITLE = {The exact blow-up rates of large solutions for semilinear
              elliptic equations},
   JOURNAL = {Nonlinear Anal.},
  FJOURNAL = {Nonlinear Analysis. Theory, Methods \& Applications. An
              International Multidisciplinary Journal},
    VOLUME = {73},
      YEAR = {2010},
    NUMBER = {11},
     PAGES = {3489--3501},
      ISSN = {0362-546X,1873-5215},
   MRCLASS = {35J91 (35B44 35J70)},
  MRNUMBER = {2718153},
       DOI = {10.1016/j.na.2010.06.069},
       URL = {https://doi.org/10.1016/j.na.2010.06.069},
}

@article{FLORICA2002,
     AUTHOR = {\c{S}t. C\^{i}rstea, Florica-Corina and R\u{a}dulescu,
              Vicen\c{t}iu D.},
     TITLE = {Existence and uniqueness of blow-up solutions for a class of
              logistic equations},
   JOURNAL = {Commun. Contemp. Math.},
  FJOURNAL = {Communications in Contemporary Mathematics},
    VOLUME = {4},
      YEAR = {2002},
    NUMBER = {3},
     PAGES = {559--586},
      ISSN = {0219-1997,1793-6683},
   MRCLASS = {35J60 (35B05 35B40 35J25)},
  MRNUMBER = {1918760},
MRREVIEWER = {Mich\`ele\ Grillot},
       DOI = {10.1142/S0219199702000737},
       URL = {https://doi.org/10.1142/S0219199702000737},
}

@article{Letelier2001,
     AUTHOR = {Garc\'{\i}a-Meli\'{a}n, J. and Letelier-Albornoz, R. and
              Sabina de Lis, J.},
     TITLE = {Uniqueness and asymptotic behaviour for solutions of
              semilinear problems with boundary blow-up},
   JOURNAL = {Proc. Amer. Math. Soc.},
  FJOURNAL = {Proceedings of the American Mathematical Society},
    VOLUME = {129},
      YEAR = {2001},
    NUMBER = {12},
     PAGES = {3593--3602},
      ISSN = {0002-9939,1088-6826},
   MRCLASS = {35J65 (35B40 35J60)},
  MRNUMBER = {1860492},
MRREVIEWER = {Varga\ Kalantarov},
       DOI = {10.1090/S0002-9939-01-06229-3},
       URL = {https://doi.org/10.1090/S0002-9939-01-06229-3},
}

@article{Zhang2020,
     AUTHOR = {Zhang, Xuemei and Feng, Meiqiang},
     TITLE = {Boundary blow-up solutions to the {M}onge-{A}mp\`ere equation:
              sharp conditions and asymptotic behavior},
   JOURNAL = {Adv. Nonlinear Anal.},
  FJOURNAL = {Advances in Nonlinear Analysis},
    VOLUME = {9},
      YEAR = {2020},
    NUMBER = {1},
     PAGES = {729--744},
      ISSN = {2191-9496,2191-950X},
   MRCLASS = {35J96 (35A01 35B40 35B44 35J60)},
  MRNUMBER = {3987859},
       DOI = {10.1515/anona-2020-0023},
       URL = {https://doi.org/10.1515/anona-2020-0023},
}

@article{Mohammed2007,
   AUTHOR = {Mohammed, Ahmed},
     TITLE = {On the existence of solutions to the {M}onge-{A}mp\`ere
              equation with infinite boundary values},
   JOURNAL = {Proc. Amer. Math. Soc.},
  FJOURNAL = {Proceedings of the American Mathematical Society},
    VOLUME = {135},
      YEAR = {2007},
    NUMBER = {1},
     PAGES = {141--149},
      ISSN = {0002-9939,1088-6826},
   MRCLASS = {35J60},
  MRNUMBER = {2280183},
       DOI = {10.1090/S0002-9939-06-08623-0},
       URL = {https://doi.org/10.1090/S0002-9939-06-08623-0},
}

@article{Materononlinear,
   AUTHOR = {Matero, Jerk},
     TITLE = {The {B}ieberbach-{R}ademacher problem for the
              {M}onge-{A}mp\`ere operator},
   JOURNAL = {Manuscripta Math.},
  FJOURNAL = {Manuscripta Mathematica},
    VOLUME = {91},
      YEAR = {1996},
    NUMBER = {3},
     PAGES = {379--391},
      ISSN = {0025-2611,1432-1785},
   MRCLASS = {35L60 (35B05 35J67)},
  MRNUMBER = {1416719},
MRREVIEWER = {Jia\ Xing\ Hong},
       DOI = {10.1007/BF02567962},
       URL = {https://doi.org/10.1007/BF02567962},
}

@article{Zhijun2018,
   AUTHOR = {Zhang, Zhijun},
     TITLE = {Large solutions to the {M}onge-{A}mp\`ere equations with
              nonlinear gradient terms: existence and boundary behavior},
   JOURNAL = {J. Differential Equations},
  FJOURNAL = {Journal of Differential Equations},
    VOLUME = {264},
      YEAR = {2018},
    NUMBER = {1},
     PAGES = {263--296},
      ISSN = {0022-0396,1090-2732},
   MRCLASS = {35J96 (35A01 35B44 35J60)},
  MRNUMBER = {3712942},
       DOI = {10.1016/j.jde.2017.09.010},
       URL = {https://doi.org/10.1016/j.jde.2017.09.010},
}

@article {ChipoXie2008,
    AUTHOR = {Chipot, M. and Xie, Y.},
     TITLE = {Some issues on the {$p$}-{L}aplace equation in cylindrical
              domains},
   JOURNAL = {Tr. Mat. Inst. Steklova},
  FJOURNAL = {Trudy Matematicheskogo Instituta Imeni V. A. Steklova},
    VOLUME = {261},
      YEAR = {2008},
     PAGES = {293--300},
      ISSN = {0371-9685,3034-1809},
      ISBN = {978-5-7846-0106-3},
   MRCLASS = {35J60 (35B40)},
  MRNUMBER = {2489714},
MRREVIEWER = {Florin\ Catrina},
       DOI = {10.1134/S0081543808020235},
       URL = {https://doi.org/10.1134/S0081543808020235},
}

@incollection {ChipotXie2004,
    AUTHOR = {Chipot, Michel and Xie, Yitian},
     TITLE = {On the asymptotic behaviour of the {$p$}-{L}aplace equation in
              cylinders becoming unbounded},
 BOOKTITLE = {Nonlinear partial differential equations and their applications},
    SERIES = {GAKUTO Internat. Ser. Math. Sci. Appl.},
    VOLUME = {20},
     PAGES = {16--27},
 PUBLISHER = {GakkB otosho, Tokyo},
      YEAR = {2004},
      ISBN = {4-7625-0429-7},
   MRCLASS = {35J60 (35B40 35J25 35J70)},
  MRNUMBER = {2087457},
}

@article {RoyLucaFiroj2021,
    AUTHOR = {Esposito, Luca and Roy, Prosenjit and Sk, Firoj},
     TITLE = {On the asymptotic behavior of the eigenvalues of nonlinear
              elliptic problems in domains becoming unbounded},
   JOURNAL = {Asymptot. Anal.},
  FJOURNAL = {Asymptotic Analysis},
    VOLUME = {123},
      YEAR = {2021},
    NUMBER = {1-2},
     PAGES = {79--94},
      ISSN = {0921-7134,1875-8576},
   MRCLASS = {35J60 (35B40)},
  MRNUMBER = {4255998},
MRREVIEWER = {Yoshitsugu\ Kabeya},
       DOI = {10.3233/asy-201626},
       URL = {https://doi.org/10.3233/asy-201626},
}

@article {Chipot2024,
    AUTHOR = {Chipot, Michel},
     TITLE = {Asymptotic behaviour of operators sum of {$p$}-{L}aplacians},
   JOURNAL = {Discrete Contin. Dyn. Syst. Ser. S},
  FJOURNAL = {Discrete and Continuous Dynamical Systems. Series S},
    VOLUME = {17},
      YEAR = {2024},
    NUMBER = {2},
     PAGES = {917--929},
      ISSN = {1937-1632,1937-1179},
   MRCLASS = {35A15 (35J60 35J66 35J92)},
  MRNUMBER = {4714534},
       DOI = {10.3934/dcdss.2023117},
       URL = {https://doi.org/10.3934/dcdss.2023117},
}

@book {Chipotbook,
    AUTHOR = {Chipot, Michel},
     TITLE = {{$\ell$} goes to plus infinity},
    SERIES = {Birkh\"auser Advanced Texts: Basler Lehrb\"ucher.
              [Birkh\"auser Advanced Texts: Basel Textbooks]},
 PUBLISHER = {Birkh\"auser Verlag, Basel},
      YEAR = {2002},
     PAGES = {viii+181},
      ISBN = {3-7643-6646-X},
   MRCLASS = {35-02 (35B40 35J60 35K55 49J40 74G65)},
  MRNUMBER = {1999898},
MRREVIEWER = {Vladimir\ Tulovsky},
       DOI = {10.1007/978-3-0348-8173-9},
       URL = {https://doi.org/10.1007/978-3-0348-8173-9},
}

@article {Jana2024,
    AUTHOR = {Jana, Purbita},
     TITLE = {Anisotropic {$p$}-{L}aplace equations on long cylindrical
              domain},
   JOURNAL = {Opuscula Math.},
  FJOURNAL = {Opuscula Mathematica},
    VOLUME = {44},
      YEAR = {2024},
    NUMBER = {2},
     PAGES = {249--265},
      ISSN = {1232-9274,2300-6919},
   MRCLASS = {35P15 (35B38 35P30)},
  MRNUMBER = {4688478},
MRREVIEWER = {Eugenio\ Massa},
}

@article {FLeoni2000,
    AUTHOR = {Leoni, Fabiana},
     TITLE = {Nonlinear elliptic equations in {${\bf R}^N$} with
              ``absorbing'' zero order terms},
   JOURNAL = {Adv. Differential Equations},
  FJOURNAL = {Advances in Differential Equations},
    VOLUME = {5},
      YEAR = {2000},
    NUMBER = {4-6},
     PAGES = {681--722},
      ISSN = {1079-9389},
   MRCLASS = {35J60 (35B65)},
  MRNUMBER = {1750115},
MRREVIEWER = {Alan\ V.\ Lair},
}

@article {SBuccheri2019,
    AUTHOR = {Buccheri, Stefano},
     TITLE = {Gradient behaviour for large solutions to semilinear elliptic
              problems},
   JOURNAL = {Ann. Mat. Pura Appl. (4)},
  FJOURNAL = {Annali di Matematica Pura ed Applicata. Series IV},
    VOLUME = {198},
      YEAR = {2019},
    NUMBER = {3},
     PAGES = {1013--1040},
      ISSN = {0373-3114,1618-1891},
   MRCLASS = {35J91 (35B44 35B45 35B65 35J25)},
  MRNUMBER = {3954403},
MRREVIEWER = {Giulio\ Romani},
       DOI = {10.1007/s10231-018-0808-y},
       URL = {https://doi.org/10.1007/s10231-018-0808-y},
}

@book {Zeidler1990,
    AUTHOR = {Zeidler, Eberhard},
     TITLE = {Nonlinear functional analysis and its applications. {II}/{B}},
      NOTE = {Nonlinear monotone operators,
              Translated from the German by the author and Leo F. Boron},
 PUBLISHER = {Springer-Verlag, New York},
      YEAR = {1990},
     PAGES = {i--xvi and 469--1202},
      ISBN = {0-387-97167-X},
   MRCLASS = {47-02 (35-01 35J60 47Hxx 58-01 65Jxx)},
  MRNUMBER = {1033498},
MRREVIEWER = {Jean\ Mawhin},
       DOI = {10.1007/978-1-4612-0985-0},
       URL = {https://doi.org/10.1007/978-1-4612-0985-0},
}

@article {DiBenedetto1983,
    AUTHOR = {DiBenedetto, E.},
     TITLE = {{$C\sp{1+\alpha }$}\ local regularity of weak solutions of
              degenerate elliptic equations},
   JOURNAL = {Nonlinear Anal.},
  FJOURNAL = {Nonlinear Analysis. Theory, Methods \& Applications. An
              International Multidisciplinary Journal},
    VOLUME = {7},
      YEAR = {1983},
    NUMBER = {8},
     PAGES = {827--850},
      ISSN = {0362-546X,1873-5215},
   MRCLASS = {35J60 (35D10)},
  MRNUMBER = {709038},
MRREVIEWER = {Bernhard\ Kawohl},
       DOI = {10.1016/0362-546X(83)90061-5},
       URL = {https://doi.org/10.1016/0362-546X(83)90061-5},
}

@article {Senoussi2009,
    AUTHOR = {Guesmia, Senoussi},
     TITLE = {Some convergence results for quasilinear parabolic boundary
              value problems in cylindrical domains of large size},
   JOURNAL = {Nonlinear Anal.},
  FJOURNAL = {Nonlinear Analysis. Theory, Methods \& Applications. An
              International Multidisciplinary Journal},
    VOLUME = {70},
      YEAR = {2009},
    NUMBER = {9},
     PAGES = {3320--3331},
      ISSN = {0362-546X,1873-5215},
   MRCLASS = {35K20 (35B40)},
  MRNUMBER = {2503078},
       DOI = {10.1016/j.na.2008.04.036},
       URL = {https://doi.org/10.1016/j.na.2008.04.036},
}

@article {RamaHarriPro2014,
    AUTHOR = {Rawat, Rama and Roy, Haripada and Roy, Prosenjit},
     TITLE = {Nonlinear elliptic eigenvalue problems in cylindrical domains
              becoming unbounded in one direction},
   JOURNAL = {Asymptot. Anal.},
  FJOURNAL = {Asymptotic Analysis},
    VOLUME = {139},
      YEAR = {2024},
    NUMBER = {3-4},
     PAGES = {245--277},
      ISSN = {0921-7134,1875-8576},
   MRCLASS = {35P30 (35J92 35P20)},
  MRNUMBER = {4858829},
       DOI = {10.3233/asy-241907},
       URL = {https://doi.org/10.3233/asy-241907},
}

@article {WangZhou2021,
    AUTHOR = {Wang, Lidan and Wang, Lihe and Zhou, Chunqin},
     TITLE = {Classification of positive solutions for fully nonlinear
              elliptic equations in unbounded cylinders},
   JOURNAL = {Commun. Pure Appl. Anal.},
  FJOURNAL = {Communications on Pure and Applied Analysis},
    VOLUME = {20},
      YEAR = {2021},
    NUMBER = {3},
     PAGES = {1241--1261},
      ISSN = {1534-0392,1553-5258},
   MRCLASS = {35J60},
  MRNUMBER = {4251794},
MRREVIEWER = {Jarmo\ J\"a\"askel\"ainen},
       DOI = {10.3934/cpaa.2021019},
       URL = {https://doi.org/10.3934/cpaa.2021019},
}

@incollection {ChipotXie2005,
    AUTHOR = {Chipot, Michel and Xie, Yitian},
     TITLE = {Asymptotic behavior of nonlinear parabolic problems with
              periodic data},
 BOOKTITLE = {Elliptic and parabolic problems},
    SERIES = {Progr. Nonlinear Differential Equations Appl.},
    VOLUME = {63},
     PAGES = {147--156},
 PUBLISHER = {Birkh\"auser, Basel},
      YEAR = {2005},
      ISBN = {978-3-7643-7249-1; 3-7643-7249-4},
   MRCLASS = {35K55 (35B10 35B40)},
  MRNUMBER = {2176709},
       DOI = {10.1007/3-7643-7384-9\_17},
       URL = {https://doi.org/10.1007/3-7643-7384-9_17},
}

@article {ChipotRougirel2008,
    AUTHOR = {Chipot, Michel and Rougirel, Arnaud},
     TITLE = {On the asymptotic behaviour of the eigenmodes for elliptic
              problems in domains becoming unbounded},
   JOURNAL = {Trans. Amer. Math. Soc.},
  FJOURNAL = {Transactions of the American Mathematical Society},
    VOLUME = {360},
      YEAR = {2008},
    NUMBER = {7},
     PAGES = {3579--3602},
      ISSN = {0002-9947,1088-6850},
   MRCLASS = {35J25 (35B40 35P15 47F05)},
  MRNUMBER = {2386237},
MRREVIEWER = {B.\ Hellwig},
       DOI = {10.1090/S0002-9947-08-04361-4},
       URL = {https://doi.org/10.1090/S0002-9947-08-04361-4},
}

@article {ChipotYeressian2008,
    AUTHOR = {Chipot, Michel and Yeressian, Karen},
     TITLE = {Exponential rates of convergence by an iteration technique},
   JOURNAL = {C. R. Math. Acad. Sci. Paris},
  FJOURNAL = {Comptes Rendus Math\'ematique. Acad\'emie des Sciences. Paris},
    VOLUME = {346},
      YEAR = {2008},
    NUMBER = {1-2},
     PAGES = {21--26},
      ISSN = {1631-073X,1778-3569},
   MRCLASS = {35J25 (35B45)},
  MRNUMBER = {2383116},
       DOI = {10.1016/j.crma.2007.12.004},
       URL = {https://doi.org/10.1016/j.crma.2007.12.004},
}

\end{document}